\documentclass[a4paper, leqno]{amsart}

\usepackage{amsmath, amsthm, mathrsfs, amssymb}

\usepackage{xcolor}

\usepackage{hyperref}
\hypersetup{
	colorlinks=true,
	linkcolor=blue,
	filecolor=magenta,      
	urlcolor=cyan,
	pdftitle={Sharelatex Example},
}

\usepackage{lineno} 

\numberwithin{equation}{section}

\newcommand{\R}{\mathbb{R}}
\newcommand{\norm}[1]{\left \lVert   #1 \right \rVert }

\newcommand{\bra}[1]{\left \langle #1 \right \rangle}  
\newtheorem{theorem}{Theorem}[section]
\newtheorem{lemma}{Lemma}[section]

\newtheorem{remark}{Remark}[section]

\title[Zakharov-Rubenchik/Benney-Roskes system]{Revisiting the Cauchy problem for the Zakharov-Rubenchik/Benney-Roskes system}

\author[H. Luong]{Hung Luong}
\address{Institute of Mathematics, VAST, 18B Hoang Quoc Viet street\\
	Cau Giay, Ha Noi, Vietnam}

\email{lthung@math.ac.vn}

\date{\today}
\begin{document}
	\begin{abstract}
		In this paper, we revisit the Cauchy problem for the Zakharov-Rubenchik/Benney-Roskes system. Our method is based on the dispersive estimates and the suitable Bourgain's spaces. We then, obtain the local well-posedness of the solution with the main component \(\psi\) belongs to \(H^1(\R^d)\) (\(d=2, 3\)) which is actually the energy space corresponding to this component. Our result also suggests a potential approach to the problem of finding exact existence time scale  for the solution of Benney-Roskes model in the context of water waves.
	\end{abstract}
		\maketitle
	
	\section{Introduction}
In this paper we revisit the Cauchy problem for the two or three-dimensional Zakharov-Rubenchik (or Benney-Roskes) system. We use the argument introduced by Bourgain (for more detail see \cite{Ginibre1997}) to obtain a better local existence result in the sense of functional spaces and of course it strengthens the results obtained in \cite{Luong2018} and \cite{Ponce2005}. Furthermore, this method suggests a potential approach to more challenge problems such as the Cauchy problem for the full dispersion Benney-Roskes system, or  finding  exact existence time scale in order to justify the Benney-Roskes system as an asymptotic model in the context of water waves.

 Let us mention that the Zakharov-Rubenchik/Benney-Roskes system (ZR/BR) is a fundamental and generic asymptotic system since it was actually derived in various physical contexts. \\
In the notations of \cite{Tzvetkov2018} (see also \cite{Passot1996} where it is used in the context of Alfvén waves in dispersive MHD), the Zakharov-Rubenchik system has the form
\begin{equation}
	\label{Intro Z-R system 1}
	\left \{
	\begin{split}
		& \psi_t - \sigma_3 \psi_x - i \delta \psi_{xx} - i \sigma_1 \Delta_{\perp}\psi + i \left \{  \sigma_2 |\psi|^2 +W (\rho+ D\phi_x) \right \} \psi = 0, \\
		& \rho_t + \Delta \phi + D (|\psi|^2)_x =0, \\
		& \phi_t + \dfrac{1}{M^2} \rho + |\psi|^2 = 0,
	\end{split}
	\right.
\end{equation}
where $\psi: \R \times \R^d \rightarrow \mathbb{C}$, $\rho,\phi : \R \times \R^d \rightarrow \R, d=2,3$ describe  the fast oscillating and, resp., acoustic type waves.

Here  $\sigma_1, \sigma_2, \sigma_3= \pm 1,$ $W>0$ measures the strength of the coupling with acoustic type waves, $M>0$ is a Mach number, $D\in \R$  is associated
to the Doppler shift due to the medium velocity and $\delta \in \R$ is a nondimensional dispersion coefficient.

When $D=0$ in \eqref{Intro Z-R system 1} the Zakharov-Rubenchik system reduces to the classical (scalar) Zakharov system (see {\it eg} Chapter V in \cite{Sulem1999}). More precisely, in the framework of \eqref{Intro Z-R system 1}, one gets
\begin{equation}
	\label{Intro Zakh}
	\left \{
	\begin{split}
		& \psi_t - \sigma_3 \psi_x - i \delta \psi_{xx} - i \sigma_1 \Delta_{\perp}\psi + i \left \{  \sigma_2 |\psi|^2 +W \rho) \right \} \psi = 0, \\
		& \rho_{tt} -\frac{1}{M^2} \Delta \rho - \Delta (|\psi|^2) =0,
	\end{split}
	\right.
\end{equation}
which is a form of the two or three dimensional Zakharov system. Note however that the second order operator in the first equation is not necessarily elliptic.

The local well-posedness in $H^s(\R^d) \times H^{s-1/2}(\R^d) \times H^{s+1/2}(\R^d)$ with $s>\frac{d}{2}, d=2,3$ for \eqref{Intro Z-R system 1} was obtained in \cite{Ponce2005} by using the local  smoothing property of the free Schr\"odinger operator after reducing the system to a quasilinear (non local) Schr\"{o}dinger equation. In \cite{Luong2018}, we assume \(\delta \sigma_1>0\) then by using method of Schochet-Weistein, we obtain the local well-posedness in $H^{s+1}(\R^2) \times H^{s}(\R^2) \times H^{s+1}(\R^2)$ with \(s>2\). Let us mention that the value of the latter result lies on the Schochet-Weistein method. In which, we transform \eqref{Intro Z-R system 1} into a symmetric nonlinear hyperbolic system, then by using an energy method, we prove the local well-posedness for \eqref{Intro Z-R system 1} perturbed by a line solitary wave. This is the first step in the framework of ``transverse stability'' problem for the line soliton.

The situation is better understood in spatial dimension one. Oliveira \cite{Oliveira2003} proved the local (thus global using the conservation laws below) well-posedness in $H^2(\R)\times H^1(\R)\times H^1(\R).$ This result was improved in \cite {Linares2009} where in particular global well-posedness was established in the energy space $H^1(\R)\times L^2(\R) \times L^2(\R).$

Let us recall these following conservation quantities with respect to \eqref{Intro Z-R system 1},
\label{conservations}
\begin{itemize}
	\item [(1)] Mass conservation:
	\begin{equation}
		\label{Mass}
		\frac{1}{2} \frac{d}{dt} \int_{\R^n} |\psi(x,t)|^2 \, dx =0.
	\end{equation}
	\item [(2)] Energy conservation:
	\begin{equation}
		\label{Energy}
		\begin{split}
			& \frac{1}{2} \frac{d}{dt} \int_{\R^n} \big (    \varepsilon |\partial_z \psi|^2 + \sigma_1 |\nabla_\perp \psi|^2 + \frac{W}{2M} \rho^2  + \frac{W}{2} |\nabla \varphi|^2 + \sigma_3 W \rho \partial_z \varphi + \frac{\sigma_2}{2} |\psi|^4  \\
			& \quad + W \rho |\psi|^2 + DW |\psi|^2 \partial_z \varphi \big ) \, dx.
		\end{split}
	\end{equation}
\end{itemize}
Those quantities suggest the energy space of \eqref{Intro Z-R system 1} is \(H^1(\R^d) \times L^2 (\R^d) \times H^1(\R^d) \) and with relevant assumptions on coefficients one gets  the existence of a global weak solution of \eqref{Intro Z-R system 1} in \cite{Ponce2005} by extending the local solution. A similar result was obtained for the perturbation of \eqref{Intro Z-R system 1} by a the so-called ``dark'' line soliton in \cite{Luong2018}.

Our goal is to establish a local well-posedness result in the energy space for \eqref{Intro Z-R system 1}, however the technical difficulty turns out that we are only able to get the \(H^1(\R^d)\) result for the first component \(\psi\) which we consider as the main part of the solution \((\psi, \rho, \phi)\). Our main result is stated in below theorem.
\begin{theorem}
	\label{theorem main}
Let \(d=2 \text{ or } 3\).	For any initial data \((\psi_0, \rho_0, \phi_0) \in H^1(\R^d) \times H^l(\R^d) \times H^{l+1}(\R^d) \), there exists \(T>0\) such that \eqref{Intro Z-R system 1} admits a unique solution \((\psi, \rho, \phi) \in C(0,T; H^1(\R^d)) \times C(0,T; H^l(\R^d)) \times C(0,T; H^{l+1}(\R^d))\). Where
\begin{align*}
	& \frac{2}{3} < l \leq 1 \text{ if } d=2, \\
	& \frac{5}{6} < l < 1 \text{ if } d=3.
\end{align*}
\end{theorem}

It is also important to mention the following versions of \eqref{Intro Z-R system 1}
\begin{equation}
	\label{Intro Z-R with epsilon}
	\left \{
	\begin{split}
		& \psi_t - \epsilon \sigma_3 \psi_x - i \epsilon \delta \psi_{xx} - i\epsilon \sigma_1 \Delta_{\perp}\psi + i \epsilon \left \{  \sigma_2 |\psi|^2 +W (\rho+ D\phi_x) \right \} \psi = 0, \\
		& \rho_t + \Delta \phi + D (|\psi|^2)_x =0, \\
		& \phi_t + \dfrac{1}{M^2} \rho + |\psi|^2 = 0,
	\end{split}
	\right.
\end{equation}
In \eqref{Intro Z-R with epsilon} the parameter \(\epsilon\) is added to the first equation as the ``model parameter'' when we consider \eqref{Intro Z-R system 1} as the Benney-Roskes system in the context of water waves problem. That leads to very important problem of proving \eqref{Intro Z-R with epsilon} is well-posed in the existence time scale \(O(1/\epsilon)\). Let us mention that the methods used in \cite{Luong2018} and \cite{Ponce2005} show the existence time scale \(O(1)\) which is not sufficient to justify \eqref{Intro Z-R with epsilon} as an asymptotic model of water waves equation. As a work in progress, we expect that with the method using in this paper we can get at least \(O(1/\epsilon^\alpha)\) with \(\alpha>0\). In our opinion, it is technically difficult and the method representing in this paper is a necessary preparation for the latter work.

The paper is organized as follows. In the Section \ref{Section setting}, we setup our problem and recall the general linear estimates using the Bourgain spaces (as in \cite{Ginibre1997}), in the latter part, we present our argument with the necessary estimates. Section \ref{Section preliminary estimate} is devoted to the preliminary estimate. In Section \ref{Section nonlinear estimate}, we present the nonlinear estimate and finalize the proof of Theorem \ref{theorem main}. Finally, we give the conclusion in Section \ref{Section Conclusion}.

Throughout this paper we use the following  notations, the others will be defined later if needed.
 \begin{itemize}
\item [1)] $\mathcal{F}, \, \mathcal{F}_t, \, \mathcal{F}_x, \, \mathcal{F}_y$ and $\mathcal{F}^{-1}$ denote  the  Fourier transform of a function in spacetime, time, space variable and the inverse Fourier transform respectively. We also use `` $\widehat{}$ ''  as the short notation of the space-time Fourier transform.
\item [2)] $H^s, H^{s,b}$ are the Sobolev's spaces with the \(L^2\) norm in space and time variables. Notation \(L_t^qL_x^r\) stands for mixed norm in space and time, \(\norm{u}_X\) is the standard norm of function \(u\) in the functional space \(X\).
\item [3)] For vector calculation, we use  $\bra{\xi} = (1 + |\xi|^2)^{1/2}$ where $\xi \in \R^d$.
\item [4)] $C$ will be a general constant unless otherwise explicitly indicated. $f \lesssim g$ (or $f \gtrsim g$) means that there exits a constant $C$ such that $f \leq C g$ (or $ f \geq C g$).
\end{itemize}

	\section{Linear estimates and the setting of problem} \label{Section setting}
	It is worth noticing that our main estimates hold in the general case of Schr\"odinger operator regardless of the sign of \(\delta\) and \(\sigma\) in \eqref{Intro Z-R system 1}. Thus,
	for simplicity, we consider  $\delta = \sigma_1=M=1, \, \sigma_3=0$ but keep the other parameters $W,D$ for futher purpose. That leads to the following system
	\begin{equation}
		\label{ZR LMS}
		\left \{
		\begin{split}
			&  i \psi_t +  \Delta \psi =  \sigma_2 |\psi|^2 \psi +  W \rho \psi +  WD \phi_x \psi, \\ 
			& \rho_t + \Delta \phi + D(|\psi|^2)_x=0,\\
			& \phi_t +  \rho + |\psi|^2=0,
		\end{split}
		\right.
	\end{equation}
	with initial data $(\psi_0, \rho_0, \phi_0)$, the space variable belongs to \(\R^d\) with \(d=2 \text{ or } 3\).
	
	
	We decouple $\rho$ and $\phi$ in the last two equations of \eqref{ZR LMS} by taking the time derivative of both equations then replace them by two wave type equations as follows
	\begin{equation}
		\label{ZR Hung}
		\left \{
		\begin{split}
			& i \psi_t +   \Delta \psi =  \sigma_2 |\psi|^2 \psi +  W \rho \psi +   WD \phi_x \psi, \\
			& \rho_{tt} - \Delta \rho = \Delta (|\psi|^2) -  D(|\psi|^2)_{xt}, \\
			& \phi_{tt} - \Delta \phi =  D(|\psi|^2)_x - (|\psi|^2)_t.
		\end{split}
		\right.
	\end{equation}
	with initial data of the form $(\psi_0, \rho_0, \phi_0, \rho_1, \phi_1)$.
	
	Set $\omega = (-\Delta)^{1/2}$, and define the positive and negative parts of $\rho, \phi$ as
	\[
	\left \{
	\begin{split}
		& \rho_{\pm} = \rho \pm i \omega^{-1} \partial_t \rho, \\
		& \phi_{\pm} = \phi \pm i \omega^{-1} \partial_t \phi.
	\end{split}
	\right.
	\]
	Then $(i\partial_t - \omega) \rho_{\pm} = \mp \omega^{-1} \square \rho$ and $\Delta = - \omega^2$, where
	\[
	\square \rho = ( \partial_t^2 - \Delta) \rho.
	\]
	Therefore, \eqref{ZR Hung} is reduced as
	\begin{equation}
		\label{old plus minus system}
		\left \{
		\begin{split}
			& i \psi_t +  \Delta \psi =   \sigma_2 |\psi|^2 \psi +   W\left ( \frac{\rho_{-} + \rho_{+}}{2} \right )\psi +  WD \left ( \frac{\phi_{+} + \phi_{-}}{2} \right)_x \psi, \\
			& (i \partial_t \mp \omega)\rho_{\pm} = \pm \omega^{-1} \Delta (|\psi|^2) \pm   D \omega^{-1} (|\psi|^2)_{xt} , \\
			& (i\partial_t \mp \omega) \phi_{\pm} = \mp   D \omega^{-1} (|\psi|^2)_x \pm \omega^{-1} (|\psi|^2)_t .
		\end{split}
		\right.
	\end{equation}
	The symbol of $\omega^{-1}$ is $1/|\xi|$ which is unbounded near $0$, so we will consider $\varphi = \phi_x$ instead of $\phi$ in \eqref{old plus minus system} in order to deal with the symbol $|\xi_1|/|\xi|$ later. That idea leads to
	\begin{equation}
		\label{plus minus system}
		\left \{
		\begin{split}
			& i \psi_t +  \Delta \psi =    \sigma_2 |\psi|^2 \psi +   W\left ( \frac{\rho_{-} + \rho_{+}}{2} \right )\psi +  WD \left ( \frac{\varphi_{+} + \varphi_{-}}{2}  \right ) \psi, \\
			& (i \partial_t \mp \omega)\rho_{\pm} = \pm \omega^{-1} \Delta (|\psi|^2) \pm   D \omega^{-1} (|\psi|^2)_{xt} , \\
			& (i\partial_t \mp \omega) \varphi_{\pm} = \mp   D \omega^{-1} (|\psi|^2)_{xx} \pm \omega^{-1} (|\psi|^2)_{xt} .
		\end{split}
		\right.
	\end{equation}

	Next, we present the general linear estimates and the construction of Bourgain spaces. Then, we rewrite the original equation into the form of an integral equation  using the Duhamel formula, introduce the cut-off equations (in time) those are crucial steps of using standard fixed point technique as for other dispersive equations.\\
	Each equation of \eqref{plus minus system} has the form
	\begin{equation}
		\label{abstract equation}
		i\partial_t u = \textbf{p}(-i\nabla) u + \textbf{q}(u),
	\end{equation}
	where $\textbf{p}$ is a real  function defined in $\R^d$ and $\textbf{q}$ is a nonlinear function. The Cauchy problem for \eqref{abstract equation} with initial data $u_0$ is rewritten as the integral equation
	\begin{equation}
		\label{abstract 2}
		u(t) = \textbf{U}(t)u_0 - i \int_0^t \textbf{U}(t-s) \textbf{q}(u(s)) ds = \textbf{U}(t)u_0 - i \textbf{U}*_R \textbf{q}(u),
	\end{equation} 
	where $\textbf{U}(t)=e^{-it\textbf{p}(-i\nabla)}$ is the unitary group defines the free evolution of \eqref{abstract equation} and $*_R$ denotes the retarded convolution in time operator.
In order to study the local (in time) Cauchy problem, we introduce the cut-off function \(\lambda(t)\).\\ $\lambda(t)\in \mathcal{C}^\infty(\R, \R^+)$ be even with $0 \leq \lambda \leq 1, \, \lambda(t)=1$ for $|t|<1, \, \lambda(t)=0$ for $|t|>2$ and let $\lambda_T=\lambda_1(t/T)$ for $0<T \leq 1$.\\
	Then \eqref{abstract 2} can be replaced by a cut-off equation
	\begin{equation}
		\label{abstract cutoff}
		u(t)= \lambda(t) \textbf{U}(t)u_0 - i \lambda_T(t) \int_0^t \textbf{U}(t-s)\textbf{q}(u(s)) ds.
	\end{equation}
	Note that \eqref{abstract cutoff} is equivalent to 
	\begin{equation}
		\label{abstract cutoff 2}
		u(t)= \lambda(t) \textbf{U}(t)u_0 - i \lambda_T(t) \int_0^t \textbf{U}(t-s)\textbf{q}(\lambda_{2T}(s)u(s)) ds,
	\end{equation}
	that is usefull for the nonlinear estimates where we want to get positive order of \(T\).
	
	We define below some general functional spaces related to the unitary group $\textbf{U}(t)$. Then, we define the functional spaces corresponding  to each equation of \eqref{plus minus system}.
	\begin{itemize}
		\item [1)] $H^{s,b}$ denotes the space time Sobolev space with the norm
		\[
		\norm{u}_{H^{s,b}} = \norm{\bra{\xi}^s \bra{\tau}^b \widehat{u}(\xi, \tau)}_2.
		\]
		\item [2)] $X^{s,b}$ denotes the Bourgain space associated to the operator $\textbf{p}(\xi)$ and the unitary group $\textbf{U}(t)$
		\[
		\norm{u}_{X^{s,b}} = \norm{\bra{\xi}^s \bra{\tau + \textbf{p}(\xi)}^b}_2.
		\]
		We can also define $X^{s,b}$ via the  equality
		\[
		\norm{u}_{X^{s,b}} = \norm{\textbf{U}(-t)u}_{H^{s,b}},
		\]
		this is the motivation of introducing the Bourgain space since it helps  eliminating the group \(\textbf{U}(t)\) on the linear term of \eqref{abstract cutoff} and \eqref{abstract cutoff 2}.
		\item [3)] An auxiliary space $Y^s$ is introduced to complete the embedding of \(X^{s,b}\) into \(C(\R, H^s(\R^d))\),
		\[
		\norm{u	}_{Y^s} = \norm{\bra{\xi}^s \bra{\tau + \textbf{p}(\xi)}^{-1} \, \widehat{u} (\xi, \tau)}_{L^2_{\xi} L^1_\tau}.
		\]
	\end{itemize}
	
	With those functional spaces we need the following linear estimates in order to evaluate the inhomogenous terms of \eqref{abstract cutoff} and \eqref{abstract cutoff 2}, for the proofs we refer to \cite{Ginibre1997}.
	\begin{lemma} 
		\label{Lemma 2.1 in [1]}
		(i) Let $  b' \leq 0 \leq b \leq b' +1  $ and $ T \leq 1 $. Then
		\begin{equation}
			\label{linear estimate 1}
			\norm{\lambda_T \, \textbf{U} *_R \textbf{q}}_{X^{s,b}} 
			 \lesssim \left ( T^{1-b+b'} \norm{\textbf{q}}_{X^{s,b'}} + T^{1/2 - b} \norm{ \textbf{q}}_{Y^s} \right ),
		\end{equation}
		
		(ii) Suppose in addition that $b' > -1/2$. Then
		\begin{equation}
			\label{linear estimate 2}
			\norm{\lambda_T \, \textbf{U} *_R \textbf{q}}_{X^{s,b}} \lesssim T^{1-b+b'} \norm{\textbf{q}}_{X^{s,b'}}.
		\end{equation}
	\end{lemma}
The last step in our argument is the embedding of \(X^{s,b}\) into \(C(\R, H^s(\R^d))\),	for  $b>1/2$ due to the Sobolev's embedding theorem, it is clear that $X^{s,b} \subset C(\R, H^s(\R^d))$. However, this is no longer true if $b \leq 1/2$ and the following result is needed.
	\begin{lemma} 
		\label{C(R,H^s)}
		Let $ \textbf{q} \in Y^s$, then $\int_0^t ds \, \textbf{U}(t-s) \textbf{q}(s) \in C(\R, H^s(\R^d))$.
	\end{lemma}

We now setup our problem \eqref{plus minus system} in the framework of \eqref{abstract cutoff}-\eqref{abstract cutoff 2}.\\
	Let  $U(t)=e^{it  \Delta}$ and $V_{\pm}(t)= e^{\mp i \omega t}$  be the unitary groups define the free evolution of \eqref{plus minus system}.\\
	Using  the cut-off functions are $\lambda(t)$ and $\lambda_T(t)$, we can rewrite \eqref{plus minus system} as follows
	\begin{equation}
		\label{cutoff psi}
		\psi_t = \lambda(t) U(t) \psi_0 - \frac{i}{2} \lambda_T(t) \int_0^t U(t-s) F(s) ds,
	\end{equation}
	
	\[
	F=F(\psi,\rho_{\pm}, \varphi_{\pm})=  \sigma_2 |\psi|^2\psi + \frac{ W}{2} (\rho_+ + \rho_-) \psi + \frac{  WD}{2}  (\varphi_+ + \varphi_-) \psi.
	\]
	\vspace{1mm}
	\begin{equation}
		\label{cutoff rho}
		\rho_{\pm} = \lambda(t) V_{\pm}(t) \rho_{\pm 0} \mp i \lambda_T(t) \int_0^t V_{\pm} (t-s) G(s) ds,
	\end{equation}
	
	\[
	G= G(\psi) = \pm \omega^{-1} \Delta (|\psi|^2) \pm  D \omega^{-1} (|\psi|^2)_{xt} \mp \omega^{-1} \rho_{\pm} .
	\]
	\vspace{1mm}
	\begin{equation}
		\label{cutoff phi}
		\varphi_{\pm}(t) = \lambda(t) V_{\pm}(t) \varphi_{\pm 0} \mp i \lambda_T(t) \int_0^t V_{\pm} (t-s) H(s) ds, 
	\end{equation}
	
	\[
	H=H(\psi) = \mp  D \omega^{-1}  (|\psi|^2)_{xx} \pm \omega^{-1} (|\psi|^2)_{t} \mp \omega^{-1} \phi_{\pm} .
	\]
	
	\vspace{1mm}
	Let  $\textbf{p}_1(\xi)=|\xi|^2, \, \textbf{p}_2(\xi)=\pm |\xi|$, we have the following Bourgain's spaces associated to $\textbf{p}_1, \textbf{p}_2$ respectively
	\begin{equation*}
			\norm{u}_{X_1^{k,b}} = \norm{ \bra{\xi}^k \bra{\tau + |\xi|^2}^b \widehat{u}(\xi,\tau)}_{L^2_{\xi,\tau}}
	\end{equation*}
	And
	\begin{align*}
		&	\norm{u}_{X_2^{k,b}} = \norm{ \bra{\xi}^k \bra{\tau \pm |\xi|}^b \widehat{u}(\xi,\tau)}_{L^2_{\xi,\tau}}, \\
		&	\norm{u}_{Y_2^k} = \norm{\bra{\xi}^k \bra{\tau \pm |\xi|}^{-1} \widehat{u}(\xi, \tau)}_{L^2_\xi(L^1_\tau)}.
	\end{align*}
	\vspace{1mm}
	We shall solve the integral equations \eqref{cutoff psi}-\eqref{cutoff phi} by a fixed point theorem with 
	\[
	\psi \text{ in } X_1^{1,b_1} ,
	\]
	\[
	\rho_{\pm} \text{ and }  \varphi_{\pm}  \text{ in } X_2^{k_2,b_2},
	\]
	here \(k_2\) is actually \(l\) in the main Theorem \ref{theorem main}, we use a symbols with indexes to precise the latter nonlinear estimates.\\
	The other symbols \(b_1, b_2\) should satisfy some ``initial'' technical conditions as follows
	\begin{align*}
		& b_1>\frac{1}{2}, \\
		& b_2 = \frac{1}{2} - \frac{k_2}{2},\\
		&  0 \leq k_2\leq 1,\\
		& c_1 + b_1 =1 \text{ and } c_2 +b_2 =1.
	\end{align*}
	The parameters \(c_1, c_2\) are defined as the parameter \(-b'\) in Lemma \ref{Lemma 2.1 in [1]}, hence they are positive.\\
	\begin{remark}
		\begin{itemize}
			\item [(i)] Firstly, we do not have parameter \(k_1\), indeed, \(k_1=1\) since we want to fix the Sobolev order of \(\psi\) as mentioned in the introduction. Although, our analysis should works in more general case of \(k_1\), we decide to fix it so that we can precise all the calculations. That actually helps if one want to deal with more challenge problem with the model parameter \(\epsilon\) involved.
			\item [(ii)] Secondly, it is worth noticing the importance of \(k_2\) or \(b_2\), so \(b_1\) will be chosen flexibly. More precisely, in our analysis, we choose \(b_2\) so that \(b_1\) can be taken satisfying the above conditions. The final conditions on \(b_1, b_2\) will be summarized in the last step of proof of \ref{theorem main} when we obtain all necessary information from the nonlinear estimates.
		\end{itemize}
	\end{remark}

	We next present all the necessary estimates following the aforementioned argument then we use the self-duality of \(L^2\) space to rewrite those estimates into integral form.\\
	Indeed, using Lemma \ref{Lemma 2.1 in [1]} leads to the following estimates:
	
	For \eqref{cutoff psi}:
	\begin{align}
		\label{nl psi 1}
		&		\norm{|\psi|^2\psi}_{X_1^{1,-c_1}} \lesssim  T^{\theta_1} \norm{\psi}_{X_1^{1,b_1}}^3, \\
		\label{nl psi 2}
		&		\norm{\rho_{\pm} \psi}_{X_1^{1,-c_1}} \lesssim  T^{\theta_2} \norm{\rho_{\pm}}_{X_2^{k_2,b_2}} \norm{\psi}_{X_1^{1,b_1}}, \\
		\label{nl psi 3}
		&		\norm{\varphi_{\pm} \psi}_{X_1^{1,-c_1}} \lesssim  T^{\theta_3} \norm{\varphi_{\pm}}_{X_2^{k_2,b_2}}  \norm{\psi}_{X_1^{1,b_1}}.
	\end{align}
	For \eqref{cutoff rho}:
	\begin{align}
		\label{nl rho 1}
		&		\norm{\omega^{-1} \Delta (|\psi|^2)}_{X_2^{k_2, -c_2}} \lesssim  T^{\theta_4} \norm{\psi}_{X_1^{1,b_1}}^2, \\
		\label{nl rho 2}
		&		\norm{\omega^{-1} (|\psi|^2)_{xt}}_{X_2^{k_2, -c_2}} \lesssim  T^{\theta_5} \norm{\psi}_{X_1^{1,b_1}}^2, \\
		\label{nl rho 3}
		&		\norm{\omega^{-1} \Delta (|\psi|^2)}_{Y_2^{k_2}} \lesssim  T^{\theta_6} \norm{\psi}_{X_1^{1,b_1}}^2, \\
		\label{nl rho 4}
		&		\norm{\omega^{-1} (|\psi|^2)_{xt}}_{Y_2^{k_2}} \lesssim  T^{\theta_{7}} \norm{\psi}_{X_1^{1,b_1}}^2
	\end{align}
	For \eqref{cutoff phi}:
	\begin{align}
		\label{nl phi 1}
		&	\norm{\omega^{-1} (|\psi|^2)_{xx}}_{X_2^{k_2, -c_2}} \lesssim  T^{\theta_{8}} \norm{\psi}_{X_1^{1,b_1}}^2, \\
		\label{nl phi 2}
		&	\norm{\omega^{-1} (|\psi|^2)_{xt}}_{X_2^{k_2, -c_2}} \lesssim  T^{\theta_{9}} \norm{\psi}_{X_1^{1,b_1}}^2, \\
		\label{nl phi 3}
		&	\norm{\omega^{-1} (|\psi|^2)_{xx}}_{Y_2^{k_2}} \lesssim  T^{\theta_{10}} \norm{\psi}_{X_1^{1,b_1}}^2, \\
		\label{nl phi 4}
		&	\norm{\omega^{-1} (|\psi|^2)_{xt}}_{Y_2^{k_2}} \lesssim  T^{\theta_{11}} \norm{\psi}_{X_1^{1,b_1}}^2.
	\end{align}
	Note that for \eqref{cutoff rho} and \eqref{cutoff phi} we need to estimate the \(Y_2^{k_2}\) norm because we are forced to choose \(b_2<\frac{1}{2}\), then the Lemma \ref{lemma ineq} is required.\\
	By the self-duality of $L^2$, it is more convenient to represent $\psi, \rho_{\pm}$ and $\varphi_{\pm}$ in the form
	\[
	\begin{split}
		& \widehat{\psi}(\xi,\tau) = \bra{\xi}^{-1} \bra{\tau +  |\xi|^2}^{-b_1} \widehat{w}(\xi,\tau), \\
		& \widehat{\overline{\psi}} (\xi, \tau)= \bra{\xi}^{-1} \bra{\tau - |\xi|^2}^{-b_1} \widehat{\overline{w}}(\xi, \tau), \\
		& \widehat{\rho_{\pm}} (\xi,\tau) = \bra{\xi}^{-k_2} \bra{\tau \pm |\xi|}^{-b_2} \widehat{u}(\xi,\tau), \\
		& \widehat{\varphi_{\pm}} (\xi,\tau) = \bra{\xi}^{-k_2} \bra{\tau \pm |\xi|}^{-b_2} \widehat{v}(\xi,\tau).
	\end{split}
	\]
	In order to estimate \eqref{nl psi 1}, we multiply $|\psi|^2\psi$ with a function in the dual space $X_1^{-1,c_1}$ which has the form $\bra{\xi}\bra{\tau + |\xi|^2}^{-c_1} \widehat{v_1}(\xi, \tau)$ where $v_1 \in L^2_{x,t}$.  This argument can be used for \eqref{nl psi 2}-\eqref{nl rho 2} and \eqref{nl phi 1}-\eqref{nl phi 2}.
	
	Similarly, to estimate $\norm{f}_{Y_2^k}$, we divide $|\widehat{f}|$ by $\bra{\tau \pm |\xi|}$ respectively,  integrate over $\tau$ for fixed $\xi$ and then take the scalar product with a generic function in $H_x^{-k}$ with Fourier transform $\bra{\xi}^k \widehat{v_3}$ and $v_3 \in L_x^2$. Using this scheme we can estimate  \eqref{nl rho 3}-\eqref{nl rho 4}  and \eqref{nl phi 3}-\eqref{nl phi 4}. 
	
	Those arguments lead to the following integrals.
	
	\noindent \textbf{Estimate} \eqref{nl psi 1}:
	\[
	\begin{split}
		I_1 & = \int \widehat{\psi^2 \overline{\psi}}(\xi, \tau) \bra{\xi}\bra{\tau +  |\xi|^2}^{-c_1} \widehat{v_1}(\xi, \tau) \, d \xi d \tau \\
		& = \int \widehat{\psi^2}(\xi_1, \tau_1) \widehat{\overline{\psi}}(\xi - \xi_1, \tau - \tau_1) \bra{\xi}\bra{\tau +  |\xi|^2}^{-c_1} \widehat{v_1}(\xi, \tau) \, d \xi d \tau d \xi_1 d \tau_1 \\
		& = \int \widehat{\psi}(\xi_2, \tau_2) \widehat{\psi}(\xi_1-\xi_2, \tau_1- \tau_2)\widehat{\overline{\psi}}(\xi - \xi_1, \tau - \tau_1) \bra{\xi}\bra{\tau +  |\xi|^2}^{-c_1} \widehat{v_1}(\xi, \tau) \\
		& \qquad \,  d \xi d \tau d \xi_1 d \tau_1 d \xi_2 d \tau_2 \\
		& = \int \frac{ \bra{\xi} \widehat{w}(\xi_2, \tau_2) \widehat{w}(\xi_1-\xi_2, \tau_1- \tau_2) \widehat{\overline{w}}(\xi - \xi_1, \tau - \tau_1) \widehat{v_1}(\xi,\tau)  }{\bra{\xi_2}\bra{\xi_1-\xi_2}\bra{\xi - \xi_1} \bra{\tau_2+  |\xi_2|^2}^{b_1}    } \\
		& \qquad \frac{\, d \xi d \tau d \xi_1 d \tau_1 d \xi_2 d \tau_2  }{\bra{\tau_1-\tau_2+ |\xi_1-\xi_2|^2}^{b_1} \bra{\tau - \tau_1- |\xi - \xi_1|^2}^{b_1} \bra{\tau +  |\xi|^2}^{c_1}}.
	\end{split}
	\]
	
	For the clear presentation, we will omit the arguments of functions on the numerator of integral and also the notation of variables. Then, 
	\[
	\begin{split}
		I_1 & = \int \frac{\bra{\xi} \widehat{w} \widehat{w} \widehat{\overline{w}} \widehat{v_1}}{\bra{\xi_2}\bra{\xi_1-\xi_2}\bra{\xi - \xi_1} \bra{\tau_2+  |\xi_2|^2}^{b_1} \bra{\tau_1-\tau_2+  |\xi_1-\xi_2|^2}^{b_1} } \\
		& \qquad \frac{1}{\bra{\tau - \tau_1-  |\xi - \xi_1|^2}^{b_1} \bra{\tau +  |\xi|^2}^{c_1}},
	\end{split}
	\]
	and \eqref{nl psi 1} is equivalent to
	\begin{equation}
		\label{nl psi 1_2}
		|I_1| \lesssim  T^{\theta_1} \norm{w}_2^3\norm{v_1}_2.
	\end{equation}
	\noindent Doing similarly, we can rewrite \eqref{nl psi 2}-\eqref{nl phi 4} as follows
	
	\noindent \textbf{Estimate} \eqref{nl psi 2}:
	\begin{equation}
		\label{nl psi 2_2}
		|I_2| \lesssim  T^{\theta_2} \norm{u}_2 \norm{w}_2 \norm{v_1}_2,
	\end{equation}
	with
	\[
	I_2= \int \frac{ \bra{\xi} \widehat{u}\widehat{w} \widehat{v_1} }{ \bra{\xi_1}^{k_2} \bra{\xi - \xi_1} \bra{\tau_1 \pm |\xi_1|}^{b_2}  \bra{\tau - \tau_1 +  |\xi - \xi_1|^2}^{b_1} \bra{\tau +  |\xi|^2}^{c_1} }.
	\]
	
	\noindent \textbf{Estimate} \eqref{nl psi 3}:
	\begin{equation}
		\label{nl psi 3_2}
		|I_3| \lesssim  T^{\theta_3} \norm{v}_2 \norm{w}_2 \norm{v_1}_2,
	\end{equation}
	with
	\[
	I_3= \int \frac{  \bra{\xi} \widehat{v}\widehat{w} \widehat{v_1} }{ \bra{\xi_1}^{k_2} \bra{\xi - \xi_1} \bra{\tau_1 \pm |\xi_1|}^{b_2} \bra{\tau - \tau_1 +  |\xi - \xi_1|^2}^{b_1} \bra{\tau + |\xi|^2}^{c_1} }.
	\]

	\noindent \textbf{Estimate} \eqref{nl rho 1}:
	\begin{equation}
		\label{nl rho 1_2}
		|I_4| \lesssim  T^{\theta_4} \norm{w}_2^2 \norm{v_2}_2,
	\end{equation}
	with
	\[
	I_4 = \int \frac{|\xi|\bra{\xi}^{k_2}\widehat{w} \widehat{\overline{w}} \widehat{v_2} }{ \bra{\xi_1} \bra{\xi -\xi_1} \bra{\tau- \tau_1 +   |\xi - \xi_1|^2}^{b_1} \bra{\tau_1 -  |\xi_1|^2}^{b_1} \bra{\tau \pm |\xi|}^{c_2}}.
	\]
	
	\noindent Estimate \eqref{nl rho 2}:
	\begin{equation}
		\label{nl rho 2_2}
		|I_5| \lesssim  T^{\theta_5} \norm{w}_2^2 \norm{v_2}_2,
	\end{equation}
	with
	\[
	I_5 = \int \frac{\xi^{(1)} \tau  \bra{\xi}^{k_2}\widehat{w} \widehat{\overline{w}} \widehat{v_2} }{|\xi| \bra{\xi_1} \bra{\xi -\xi_1} \bra{\tau- \tau_1 +  |\xi - \xi_1|^2}^{b_1} \bra{\tau_1 - |\xi_1|^2}^{b_1} \bra{\tau \pm |\xi|}^{c_2}}.
	\]

	Estimate \eqref{nl rho 3}:
	\begin{equation}
		\label{nl rho 3_2}
		|I_6| \lesssim  T^{\theta_6} \norm{w}_2^2 \norm{v_3}_2
	\end{equation}
	with
	\[
	I_6 = \int \frac{ |\xi| \bra{\xi}^{k_2} \widehat{w}\widehat{\overline{w}} \widehat{v_3}  }{ \bra{\xi_1} \bra{\xi- \xi_1} \bra{\tau - \tau_1 +  |\xi - \xi_1|^2}^{b_1} \bra{\tau_1 -  |\xi_1|^2}^{b_1} \bra{\tau \pm |\xi|}  }.
	\]
	
	\noindent Estimate \eqref{nl rho 4}:
	\begin{equation}
		\label{nl rho 4_2}
		|I_{7}| \lesssim  T^{\theta_{7}} \norm{w}_2^2 \norm{v_3}_2,
	\end{equation}
	with
	\[
	I_{7} = \int \frac{ \xi^{(1)} \tau \bra{\xi}^{k_2} \widehat{w}\widehat{\overline{w}} \widehat{v_3}  }{ |\xi| \bra{\xi_1} \bra{\xi- \xi_1} \bra{\tau - \tau_1 +  |\xi - \xi_1|^2}^{b_1} \bra{\tau_1 -  |\xi_1|^2}^{b_1} \bra{\tau \pm |\xi|}  },
	\]
	here, \(\xi^{(1)}\) denotes the first component of vector \(\xi\) in \(\R^2\).
	
	\noindent Estimate \eqref{nl phi 1}:
	\begin{equation}
		\label{nl phi 1_2}
		|I_{8}| \lesssim  T^{\theta_{8}} \norm{w}_2^2 \norm{v_2}_2,
	\end{equation}
	with
	\[
	I_{8}= \int \frac{(\xi^{(1)})^2 \bra{\xi}^{k_2}\widehat{w} \widehat{\overline{w}} \widehat{v_2} }{|\xi| \bra{\xi_1} \bra{\xi -\xi_1} \bra{\tau- \tau_1 +  |\xi - \xi_1|^2}^{b_1} \bra{\tau_1 -  |\xi_1|^2}^{b_1} \bra{\tau \pm |\xi|}^{c_2}}.
	\]
	\noindent Estimate \eqref{nl phi 2}
	\begin{equation}
		\label{nl phi 2_2}
		|I_{9}| \lesssim  T^{\theta_{9}} \norm{w}_2^2 \norm{v_2}_2,
	\end{equation}
	with 
	\[
	I_{9}= \int \frac{ \xi^{(1)}  \tau \bra{\xi}^{k_2}\widehat{w} \widehat{\overline{w}} \widehat{v_2}  }{ |\xi| \bra{\xi_1} \bra{\xi -\xi_1} \bra{\tau- \tau_1 +  |\xi - \xi_1|^2}^{b_1} \bra{\tau_1 -  |\xi_1|^2}^{b_1} \bra{\tau \pm |\xi|}^{c_2}}.
	\]
	\noindent Estimate \eqref{nl phi 3}:
	\begin{equation}
		\label{nl phi 3_2}
		|I_{10}| \lesssim  T^{\theta_{10}} \norm{w}_2^2 \norm{v_3}_2,
	\end{equation}
	with
	\[
	I_{10} = \int \frac{ (\xi^{(1)})^2 \bra{\xi}^{k_2} \widehat{w}\widehat{\overline{w}} \widehat{v_3}  }{ |\xi| \bra{\xi_1} \bra{\xi- \xi_1} \bra{\tau - \tau_1 +  |\xi - \xi_1|^2}^{b_1} \bra{\tau_1 -  |\xi_1|^2}^{b_1} \bra{\tau \pm |\xi|}  }.
	\]
	\noindent Estimate \eqref{nl phi 4}:
	\begin{equation}
		\label{nl phi 4_2}
		|I_{11}| \lesssim  T^{\theta_{11}} \norm{w}_2^2 \norm{v_3}_2,
	\end{equation}
	with 
	\[
	I_{11} = \int \frac{ \xi^{(1)} \tau \bra{\xi}^{k_2} \widehat{w}\widehat{\overline{w}} \widehat{v_3}  }{ |\xi| \bra{\xi_1} \bra{\xi- \xi_1} \bra{\tau - \tau_1 +  |\xi - \xi_1|^2}^{b_1} \bra{\tau_1 -  |\xi_1|^2}^{b_1} \bra{\tau \pm |\xi|}  }.
	\]
	\section{Preliminary estimates} \label{Section preliminary estimate}
		In this section, to prepare for the proofs of \eqref{nl psi 1_2}-\eqref{nl phi 4_2}, we recall the Strichartz estimates and some elementary inequalities.
	\begin{lemma} \label{lemma ineq schr} (Strichartz estimate, \cite{Ginibre1997})\\
		 Let $b_0 > 1/2$, let $a \geq 0$, $a' \geq 0$, let $0 \leq \gamma \leq 1$. Assume in addition that $(1 - \gamma) a \leq b_0$ and $\gamma a \leq a '$. Let $0 < \eta \leq 1$ and define $q$ and $r$ by
		\begin{gather}
			\label{strichartz 1}   \frac{2}{q} = 1 - \frac{\eta(1 - \gamma) a}{b_0} \\
			\label{strichartz 2} \delta(r)=\frac{d}{2} - \frac{d}{r} = \frac{(1- \eta) (1 -\gamma) a }{b_0}.
		\end{gather}
		Let $v \in L^2$ be such that $\mathcal{F}^{-1} (\bra{\tau + |\xi|^2}^{-a'} \widehat{v})$ has support in $|t| \leq CT$. Then
		\begin{equation}
			\label{strichartz 3} \norm{\mathcal{F}^{-1} (\bra{\tau + |\xi|^2 }^{-a} |\widehat{v}|)}_{L_t^q L_x^r} \leq C T^\theta \norm{v}_2,
		\end{equation}
		\begin{equation}
			\label{lemma ineq defi theta}
			\theta = \gamma a \left(   1- [a'-1/2]_+/a' \right)
		\end{equation}
		We recall that 
		\begin{equation*}
			[\lambda]_+ =  \left \{ \begin{split}
				& \lambda \text{ if } \lambda>0, \\
				& \varepsilon>0 \text{ if } \lambda=0,\\
				& 0 \text{ if } \lambda<0.
			\end{split}
			\right.
		\end{equation*}
		
		For the wave equation, i.e. $\sigma = \tau \pm |\xi|$, we only consider the special cases of \eqref{strichartz 3} when $\eta=1$ and $r=2$. So, $q$ is defined by
		\begin{equation}
			\label{lemma ineq wave defi q}
			\frac{2}{q}= 1 - (1-\gamma) \frac{a}{b_0}.
		\end{equation}
		Let $v\in L^2$ be such that $\mathcal{F}^{-1} (\bra{\sigma}^{-a'} |\widehat{v}|)$ has support in $|t| \leq CT$. Then
		\begin{equation}
			\label{lemma ineq wave 4}
			\norm{\mathcal{F}^{-1} ( \bra{\sigma}^{-a} |\widehat{v}|)}_{L^q_tL^2_x} \leq CT^{\theta} \norm{v}_2
		\end{equation}
		with $\theta \geq 0$. Note that $\theta = 0$ if and only if $a=0$ or $\gamma =0$.
	\end{lemma}
	\begin{remark}
		Those estimates together with the cut-off procedure in \eqref{cutoff psi}-\eqref{cutoff phi} ensure the appearance of \(T\).
	\end{remark}
	

	\begin{lemma}(Symbolic inequalities) \label{lemma ineq}
		Let $\xi, \xi_1, \xi_2$ be vectors in $\R^d$ $(d=2,3)$ and $\tau, \tau_1 \in \R$ then we have the following inequalities.
		
		\noindent	i) For all $\xi, \xi_1,\xi_2$, we have
		\begin{equation}
			\label{ineq 1}
			\bra{\xi} \leq \bra{\xi_2} + \bra{\xi_1-\xi_2} + \bra{\xi - \xi_1}. 
		\end{equation}
		
		\vspace*{1mm}
		\noindent	ii) If  $|\xi|>2|\xi - \xi_1|$, then 
		\begin{equation}
			\label{ineq 2}
			 \bra{\xi}^2 \lesssim \bra{\tau_1 \pm |\xi_1|} + \bra{\tau-\tau_1 + |\xi - \xi_1|^2} + \bra{\tau + |\xi|^2}.
		\end{equation}
		
		\vspace*{1mm}
		\noindent	iii) For all $\xi, \xi_1, \xi_2$ we have
		\begin{equation}
			\label{ineq 3}
			\bra{\xi}^2 \lesssim   \bra{\tau - \tau_1 + |\xi - \xi_1|^2} + \bra{\tau_1 - |\xi_1|^2} + \bra{\tau \pm |\xi|} 
		\end{equation}
		holds.
		
		\vspace*{1mm}
		\noindent	iv) For all $\tau, \tau_1, \xi, \xi_1$, we have
		\begin{equation}
			\label{ineq 4}
			\bra{\xi} \bra{\tau \pm  |\xi|^2}^{1/2} \gtrsim |\tau|^{1/2},
		\end{equation}
		then, as a corollary
		\begin{equation}
			\label{ineq 5}
			\bra{\xi_1} \bra{\xi - \xi_1} \bra{\tau - \tau_1 + |\xi - \xi_1|^2}^{1/2} \bra{\tau_1 - |\xi_1|^2}^{1/2} \gtrsim |\tau|^{1/2}.
		\end{equation}
	\end{lemma}
	
	\begin{proof}
		i) This inequality follows directly Cauchy-Schwartz inequality.
		
		\vspace*{1mm}
		ii) 	If $|\xi| \leq 4 $, then the estimate is obvious. Let $|\xi|>4,$ then we have
		\[
		|\tau + |\xi|^2| + |\tau - \tau_1 + |\xi - \xi_1|^2| + |\tau_1 \pm |\xi_1|| \geq ||\xi|^2 - ||\xi- \xi_1|^2 \mp |\xi_1|||.	  
		\]
		Moreover,
		\[
		|\xi|^2 - ||\xi- \xi_1|^2 \mp |\xi_1|| \geq |\xi|^2 - (|\xi - \xi_1|^2 + |\xi_1|)
		\]
		combining with 
		\[
		|\xi - \xi_1|\leq \frac{|\xi|}{2},\
		\]
		and
		\[
		|\xi_1| = |\xi_1 - \xi + \xi| \leq   \frac{3}{2} |\xi|,
		\]
		we have
		\[
		|\xi|^2 - ||\xi- \xi_1|^2 \mp |\xi_1|| \geq \frac{3}{4} |\xi|^2 - \frac{3}{2}|\xi| = \frac{3}{8} |\xi|(|\xi| -4) + \frac{3}{8} |\xi|^2 \geq \frac{3}{8}  |\xi|^2.
		\]
		That completes the proof of \eqref{ineq 2}.
		
		\vspace*{1mm}
		iii) We use the similar argument as in previous part, if $|\xi| < C$ for a general constant $C$ then \eqref{ineq 3} holds. That means in next step we can assume that $|\xi|$ as large as we need.
		
		By	Using the triangle inequality we have
		\[
		\bra{\tau - \tau_1 + |\xi - \xi_1|^2} + \bra{\tau_1 - |\xi_1|^2} + \bra{\tau \pm |\xi|}  \gtrsim 3 + \big ||\xi-\xi_1|^2 - |\xi_1|^2 \mp |\xi| \big |.
		\]
		If $|\xi| \geq 3 |\xi_1|$ or $|\xi_1|\leq \frac{1}{3} |\xi|$ then 
		\[
		|\xi - \xi_1| - |\xi_1| \geq |\xi| - 2 |\xi_1| \geq \frac{1}{3}|\xi|,
		\]
		so
		\[
		\begin{split}
			\big ||\xi-\xi_1|^2 - |\xi_1|^2 \mp |\xi| \big | & \geq \big ||\xi - \xi_1|^2 - |\xi_1|^2 \big | - |\xi| \\
			& = \big | |\xi - \xi_1| - |\xi_1| \big | \big (|\xi- \xi_1| + |\xi_1|\big ) - |\xi| \\
			& \geq \frac{1}{3} |\xi|^2 - |\xi|\\
			& \geq \frac{1}{6}  |\xi|^2 \quad  (\text{ if } |\xi| \geq 6).
		\end{split}
		\]
		Then, if $|\xi|>6$ we have 
		\[
		\bra{\tau - \tau_1 + |\xi - \xi_1|^2} + \bra{\tau_1 - |\xi_1|^2} + \bra{\tau \pm \bra{\xi}}  \gtrsim \bra{\xi}^2.
		\]
		\vspace*{2mm}
		
		If $|\xi | < 3 |\xi_1|$ then we continue to split the domain of $\xi$ and $\xi_1$.\\
		\vspace*{1mm}
		\noindent	If $\frac{1}{4} |\xi| \leq  |\xi - \xi_1|$ then 
		\[
		\frac{1}{9}  |\xi|^2 + \frac{1}{16}  |\xi|^2 \mp |\xi|	 \leq \tau - \tau_1 +  |\xi - \xi_1|^2 + \tau_1 +  |\xi_1|^2 - (\tau \pm |\xi|),
		\]
		so, for $|\xi| > 16$
		\[
		\frac{1}{9}  |\xi|^2 \leq  |\tau - \tau_1 +  |\xi - \xi_1|^2| + |\tau_1 +  |\xi_1|^2| + |\tau \pm |\xi||,
		\]
		or equivalently, \eqref{ineq 3} holds.
		
		\vspace*{1mm}
		If $\frac{1}{4} |\xi| >  |\xi - \xi_1|$ then
		\[
		|\xi_1| - |\xi - \xi_1| \geq |\xi_1| - \frac{1}{4} |\xi|,
		\]
		note that we are considering the case: $|\xi_1| > \frac{1}{3} |\xi|$, so
		\[
		|\xi_1| - |\xi - \xi_1| > \frac{1}{12} |\xi|.
		\]
		Let observe again
		\[
		\begin{split}
			\big | |\xi_1|^2 - |\xi - \xi_1|^2 \mp |\xi| \big |  & \geq \big | |\xi_1|^2 - |\xi - \xi_1|^2   \big | - |\xi| \\
			& = \big | |\xi - \xi_1| - |\xi_1| \big |  \big (|\xi- \xi_1| + |\xi_1| \big ) - |\xi|\\
			& = \big (|\xi_1| - |\xi - \xi_1| \big ) \big (|\xi- \xi_1| + |\xi_1| \big ) - |\xi|\\
			& > \frac{1}{12} |\xi|^2 - |\xi|\\
			& > \frac{1}{24}  |\xi|^2 \quad (\text{if } |\xi|>24).
		\end{split}
		\]

		Finally, if $|\xi| > Max (M_1,M_2)$ then \eqref{ineq 3} holds.
		
		\vspace*{1mm}
		iv) We first prove \eqref{ineq 4}. Using the Cauchy-Schwartz inequality it is not difficult to see that
		\[
		\begin{split}
			\bra{\xi}^2 \bra{\tau \pm  |\xi|^2} & = \sqrt{(1+ |\tau \pm |\xi|^2|^2) (1+|\xi|^2)^2} \\
			& \gtrsim \sqrt{(1+ |\tau \pm |\xi|^2|^2) (1+ |\xi|^4)} \\
			& \gtrsim \bra{\tau}^{1/2}.
		\end{split}
		\]
		That is \eqref{ineq 4} and \eqref{ineq 5} follows  directly.
	\end{proof}

	\section{Nonlinear estimates} \label{Section nonlinear estimate}
	In this section, we are going to prove the nonlinear estimates \eqref{nl psi 1_2}-\eqref{nl phi 4_2} and finish the proof of the main theorem. Our goal is obtaining positive order of \(T\) so that \eqref{ZR Hung} can be solved locally in time. The argument relies on the fixed-point technique which  is similar as  in \cite{Luong2018a} and \cite{Ginibre1997}. We need to estimates all the nonlinear terms in cut-off integral equations  \eqref{cutoff psi}, \eqref{cutoff rho}, \eqref{cutoff phi}, or more precisely the estimates from \eqref{nl psi 1}-\eqref{nl phi 4}.
	The proof is organized as follows,
	\begin{itemize}
		\item [(i)] First, in \ref{subsection nonlinear estimates}, We prove the estimates for $I_1,I_2,I_4$ and $I_5$.\\
	The following pairs of integrals have similar form then their proofs are essentially the same: $I_2$ and $I_3$, $I_4$ and $I_8$, $I_5$ and $I_9$.\\
		The estimates for $I_6$,  $I_7, I_{10}$ and $I_{11}$ can be deduced directly from the estimates for $I_4, I_5, I_8$ and $I_9$ respectively.
		
		\item [(ii)] Finally, in \ref{subsection proof}, we summarize the condition of parameters \(b_1,b_2\) those define the order of Sobolev spaces.
	\end{itemize}

	\vspace{2mm}
	\subsection{Nonlinear estimates} \label{subsection nonlinear estimates}
	First, let consider $I_1$, using \eqref{ineq 1}, Plancherel identity and the H\"older inequality we have
	\begin{equation}
		\label{I1}
		\begin{split}
			|I_1| & \leq \int \frac{(\bra{\xi_2} + \bra{\xi_1-\xi_2} + \bra{\xi - \xi_1}) |\widehat{w}| |\widehat{w}| |\widehat{\overline{w}}| |\widehat{v_1}|}{\bra{\xi_2} \bra{\xi_1-\xi_2} \bra{\xi - \xi_1} } \\
			& \qquad \frac{1}{\bra{\tau_2+ |\xi_2|^2}^{b_1} \bra{\tau_1-\tau_2+ |\xi_1-\xi_2|^2}^{b_1}   \bra{\tau - \tau_1- |\xi - \xi_1|^2}^{b_1} \bra{\tau + |\xi|^2}^{c_1}} \\
		\end{split}
	\end{equation}
	Using the H\"older inequality and the Plancherel identity, the right hand side (RHS) of \eqref{I1} is bounded by the terms of the following form
	\[
	\begin{split}
		& \norm{\mathcal{F}^{-1} ( \bra{\xi}^{-1} \bra{\tau + |\xi|^2}^{-b_1} |\widehat{w}|)}_{L_t^{q_1}L_x^{r_1}}^2   \norm{\mathcal{F}^{-1} (\bra{\tau + |\xi|^2}^{-b_1} |\widehat{w}|)}_{L_t^{q_1}L_x^{r_2}} \\ 
		& \quad \norm{\mathcal{F}^{-1} (\bra{\tau + |\xi|^2}^{-c_1} |\widehat{v_1}|)}_{L_t^{q_3}L_x^{r_3}},
	\end{split}
	\]
	provided that
	\begin{align}
		& \label{Holder I1 1} \frac{3}{q_1} + \frac{1}{q_3} =1,\\
		& \label{Holder I1 2} 2 \delta(r_1) + \delta(r_2)+ \delta(r_3)=d,
	\end{align}
	we remind that \(\delta(r):= \frac{d}{2} - \frac{d}{r}\).
	
	The two terms: $\norm{\mathcal{F}^{-1} (\bra{\tau +  |\xi|^2}^{-b_1} |\widehat{w}|)}_{L_t^{q_1}L_x^{r_2}}$ and \\ $\norm{\mathcal{F}^{-1} (\bra{\tau + |\xi|^2}^{-c_1} |\widehat{v_1}|)}_{L_t^{q_3}L_x^{r_3}}$ are estimated in terms of  $\norm{w}_2$ and $\norm{v_1}_2$ via Lemma \ref{lemma ineq schr} with the following constrains:
	\begin{align*}
		&  \frac{2}{q_1} = 1- \eta(1-\gamma) \frac{b_1}{b_0},\\
		& \delta(r_2) = (1-\eta)(1-\gamma) \frac{b_1}{b_0},\\
		& \frac{2}{q_3} = 1- \eta(1-\gamma) \frac{c_1}{b_0}, \\
		& \delta(r_3) = (1-\eta)(1-\gamma) \frac{c_1}{b_0}.
	\end{align*}

	For  $ \norm{\mathcal{F}^{-1} ( \bra{\xi}^{-1} \bra{\tau +  |\xi|^2}^{-b_1} |\widehat{w}|)}_{L_t^{q_1}L_x^{r_1}}  $, we first use the Sobolev's embedding theorem
	\[
	W^{1,r_2}(\R^d) \hookrightarrow L_x^{r_1} (\R^d)  \text{ if } 1 \geq \delta(r_1) - \delta(r_2),
	\]
	then it can be bounded by $\norm{w}_2$ using Lemma \ref{lemma ineq schr} as in previous step.

	Therefore, \eqref{Holder I1 1} and \eqref{Holder I1 2} lead to
	\begin{align}
		& \label{I1 6} \frac{\eta(1-\gamma) (2b_1+1)}{2b_0} =1, \\
		& \label{I1 7} \frac{(1-\eta)(1-\gamma)(2b_1+1)}{b_0} \geq d-2.
	\end{align}
	Combining \eqref{I1 6} and \eqref{I1 7} we obtain
	\[
	\eta \leq \frac{2}{d},
	\]
	that suggests us to take $\eta = \frac{2}{d}$ and then 
	\[
	1-\gamma = \frac{db_0}{2b_1 +1}.
	\]
	It remains to choose $b_0$ such that $b_0 >1/2$, $(1-\gamma)b_1 \leq b_0$ and $0 \leq 1-\gamma \leq 1$.\\
	If we choose $b_0=b_1$ then we only need to verify that $1-\gamma < 1$. It is not difficult to see that holds for $d=2,3$.
	
	Therefore, we have
	\begin{equation}
		\label{final I1} |I_1| \lesssim  T^{\theta_1} \norm{w}_2^3 \norm{v_1}_2,
	\end{equation}
	where
	\begin{equation}
		\label{t I1}  \theta_1= \left(1- \frac{db_1}{2b_1+1}\right) \left(\frac{5}{2} - b_1\right).
	\end{equation}
	and \(\theta_1>0\).

	\textbf{Estimate $I_2$}. Using the Schwartz inequality, we have
	\[
	\begin{split}
		I_2 & = \int \frac{ \bra{\xi} \widehat{u}\widehat{w} \widehat{v_1} }{ \bra{\xi_1}^{k_2} \bra{\xi - \xi_1} \bra{\tau_1 \pm |\xi_1|}^{b_2}  \bra{\tau - \tau_1 + |\xi - \xi_1|^2}^{b_1} \bra{\tau + |\xi|^2}^{c_1} } \\
		& \leq \int\frac{(\bra{\xi_1}^{k_2} + \bra{\xi - \xi_1}^{k_2}) \bra{\xi}^{1-k_2} \widehat{u}\widehat{w} \widehat{v_1} }{ \bra{\xi_1}^{k_2} \bra{\xi - \xi_1} \bra{\tau_1 \pm |\xi_1|}^{b_2}  \bra{\tau - \tau_1 + |\xi - \xi_1|^2}^{b_1} \bra{\tau + |\xi|^2}^{c_1} } \\
		& = \int \frac{ \bra{\xi}^{1-k_2} \widehat{u}\widehat{w} \widehat{v_1}}{\bra{\xi - \xi_1} \bra{\tau_1 \pm |\xi_1|}^{b_2}  \bra{\tau - \tau_1 + |\xi - \xi_1|^2}^{b_1} \bra{\tau + |\xi|^2}^{c_1}} \\
		& \quad + \int \frac{\bra{\xi}^{1-k_2} \widehat{u}\widehat{w} \widehat{v_1}}{\bra{\xi_1}^{k_2} \bra{\xi - \xi_1}^{1-k_2} \bra{\tau_1 \pm |\xi_1|}^{b_2}  \bra{\tau - \tau_1 + |\xi - \xi_1|^2}^{b_1} \bra{\tau + |\xi|^2}^{c_1}} \\
		& = I_{21} + I_{22} + I_{23} + I_{24}.
	\end{split}
	\]
	Where
	
	\begin{align*}
		& I_{21} = \int_{|\xi| \leq 2 |\xi - \xi_1|} \frac{\widehat{u}\widehat{w} \widehat{v_1}}{ \bra{\xi -\xi_1}^{k_2} \bra{\tau_1 \pm |\xi_1|}^{b_2}  \bra{\tau - \tau_1 + |\xi - \xi_1|^2}^{b_1} \bra{\tau + |\xi|^2}^{c_1}}, \\
		&  I_{22} = \int_{|\xi |> 2|\xi - \xi_1|} \frac{\bra{\xi}^{2b_2} \widehat{u}\widehat{w} \widehat{v_1}}{\bra{\xi - \xi_1} \bra{\tau_1 \pm |\xi_1|}^{b_2}  \bra{\tau - \tau_1 + |\xi - \xi_1|^2}^{b_1} \bra{\tau + |\xi|^2}^{c_1}}, \\
		& I_{23} = \int_{|\xi| \leq 2 |\xi - \xi_1|} \frac{\widehat{u}\widehat{w} \widehat{v_1}}{\bra{\xi_1}^{k_2} \bra{\tau_1 \pm |\xi_1|}^{b_2}  \bra{\tau - \tau_1 + |\xi - \xi_1|^2}^{b_1} \bra{\tau + |\xi|^2}^{c_1} }, \\
		& I_{24} = \int_{|\xi| > 2|\xi - \xi_1|} \frac{\bra{\xi}^{2b_2} \widehat{u}\widehat{w} \widehat{v_1}}{\bra{\xi_1}^{k_2} \bra{\xi - \xi_1}^{2b_2} \bra{\tau_1 \pm |\xi_1|}^{b_2} \bra{\tau - \tau_1 + |\xi - \xi_1|^2}^{b_1} \bra{\tau + |\xi|^2}^{c_1} }.
	\end{align*}
	
	\vspace{0.2cm}
	\textbf{Estimate $I_{21}$:} Using the H\"older inequality we obtain that
	\begin{equation} \label{I21}
		\begin{split}
			|I_{21}| & \leq \norm{\mathcal{F}^{-1}\left(\bra{\xi}^{-k_2} \bra{\tau +  |\xi^2|}^{-c_1} |\widehat{v_1}| \right)}_{L_t^{q_1} L_x^{r_1}} \norm{\mathcal{F}^{-1} \left( \bra{\tau +  |\xi|^2}^{-b_1} |\widehat{w}| \right)}_{L_t^{q_2}L_x^{r_2}} \\
			& \quad \norm{\mathcal{F}^{-1} \left(  \bra{\tau \pm |\xi|}^{-b_2}  |\widehat{u}|\right)}_{L_t^{q_3} L_x^{2}},
		\end{split}
	\end{equation}
	provided that
	
	\begin{align} 
		&\frac{1}{q_1} + \frac{1}{q_2} + \frac{1}{q_3}=1,  \label{Holder I21 1}\\
		& \frac{1}{r_1} + \frac{1}{r_2} = \frac{1}{2} \text{ or } \delta(r_1) + \delta(r_2) = \frac{d}{2}. \label{Holder I21 2}
	\end{align}
	
	Using the Sobolev's embedding theorem, we know that 
	
	\begin{equation} \label{Sobolev I21}
		W^{k_2,r_1'} \hookrightarrow L_x^{r_1} \text{ if  } k_2 \geq \delta(r_1) - \delta(r_1').
	\end{equation}
	
	The first term of \eqref{I21} is bounded by $\norm{\mathcal{F}^{-1} \left( \bra{\tau +  |\xi|^2}^{-c_1} |\widehat{v_1}|\right)}_{L_t^{q_1} L_x^{r_1'}}$.  Then, this term and the last two terms of \eqref{I21} can be estimated by using Lemma \ref{lemma ineq schr}, provided that
	
	\begin{align*}
		& \frac{2}{q_1} = 1- \eta(1-\gamma) \frac{c_1}{b_0}, \\
		& \delta(r_1')= (1-\eta)(1-\gamma)\frac{c_1}{b_0},\\
		& \frac{2}{q_2} = 1- \eta(1-\gamma) \frac{b_1}{b_0}, \\
		& \delta(r_2) = (1-\eta)(1-\gamma) \frac{b_1}{b_0}, \\
		& \frac{2}{q_3} = 1- (1-\gamma) \frac{b_2}{b_0}.
	\end{align*}
	Therefore the restrictions \eqref{Holder I21 1}-\eqref{Holder I21 2} and  \eqref{Sobolev I21} become
	\begin{align}
		& (1-\gamma) \frac{b_2 + \eta}{b_0} =1,  \label{Condition I21 1}\\
		&\frac{ (1-\eta)(1-\gamma) }{b_0} \geq \frac{d}{2}  +2b_2 -1 .  \label{Condition I21 2}
	\end{align}
	From \eqref{Condition I21 1}, \eqref{Condition I21 2} we have that
	\[
	\eta \leq \frac{1+b_2}{d/2+ 2b_2} -b_2
	\]
	that suggests us to take $\eta = \frac{1+b_2}{d/2+ 2b_2} -b_2$. Indeed, for $d=2,3$ we can verify that $0 \leq \eta \leq 1$, then $1-\gamma = \frac{b_0(d+4b_2)}{2+2b_2}$.\\
	If we choose $b_0=b_1$ then it remains to ensure that $1-\gamma  < 1$, or equivalently
	\begin{equation}
		\label{auxi 1}	 b_1 < \frac{2+2b_2}{d+4b_2}.
	\end{equation}
	It is not difficult to see that for $b_2<\frac{1}{2}$ the right hand side of \eqref{auxi 1} is always strictly greater than $\frac{1}{2}$. Thus, in general the assumption $b_1>\frac{1}{2}$ makes sense. However, we will need to combine \eqref{auxi 1} with later constrains from other estimates to conclude on the final condition of $b_1$.
	
	Therefore, we have
	\begin{equation}
		\label{final I21}
		|I_{21}| \lesssim  T^{\theta_{21}} \norm{v_1}_2 \norm{w}_2 \norm{u}_2,
	\end{equation}
	where
	\begin{equation}
		\label{t I21}  \theta_{21} = (1-\frac{b_1(d+4b_2)}{2+2b_2}) (b_2+\frac{3}{2} -b_1)>0.
	\end{equation}

	\textbf{Estimate $I_{22}$:} Using \eqref{ineq 2} we see that If $|\xi| > 2|\xi - \xi_1|$ then
	\[
	 \bra{\xi}^{2b_2} \lesssim \bra{\tau_1 \pm |\xi_1|}^{b_2}  + \bra{\tau - \tau_1 + |\xi - \xi_1|^2}^{b_2} + \bra{\tau +  |\xi|^2}^{b_2}.
	\]
	That implies
	\[
	|I_{22}| \leq I_{221} + I_{222}+ I_{223},
	\]
	where
	\begin{align*}
		& I_{221} = \int_{|\xi |> 2|\xi - \xi_1|} \frac{|\widehat{u}| |\widehat{w}| |\widehat{v_1}| }{ \bra{\xi - \xi_1} \bra{\tau - \tau_1+  |\xi - \xi_1|^2}^{b_1} \bra{\tau + |\xi|^2}^{c_1}  }, \\
		& I_{222} = \int_{|\xi |> 2|\xi - \xi_1|} \frac{ |\widehat{u}| |\widehat{w}| |\widehat{v_1}| }{\bra{\xi - \xi_1} \bra{\tau_1 \pm |\xi_1|}^{b_2} \bra{\tau - \tau_1 +  |\xi - \xi_1|^2}^{b_1 - b_2} \bra{\tau +  |\xi|^2}^{c_1}}, \\
		& I_{223} = \int_{|\xi |> 2|\xi - \xi_1|} \frac{ |\widehat{u}| |\widehat{w}| |\widehat{v_1}|}{\bra{\xi - \xi_1}  \bra{\tau_1 \pm |\xi_1|}^{b_2}  \bra{\tau - \tau_1 + |\xi - \xi_1|^2}^{b_1}  \bra{\tau + |\xi|^2}^{c_1-b_2}}.
	\end{align*}
	\vspace{2mm}
	
	For $I_{221}$, by using the H\"older inequality we have
	\vspace{2mm}
	\begin{equation}
		\begin{split}
			\label{I221} I_{221} & \leq  \norm{\mathcal{F}^{-1} \left(  \bra{\xi}^{-1} \bra{\tau +  |\xi|^2}^{-b_1} |\widehat{w}| \right)}_{L_t^{q_1}L_x^{r_1}} \norm{\mathcal{F}^{-1} \left(   \bra{\tau +  |\xi|^2}^{-c_1} |\widehat{v_1}| \right)}_{L_t^{q_2}L_x^{r_2}} \\
			& \quad \norm{\mathcal{F}^{-1}(|\widehat{u}|)}_{L_t^2L_x^2}
		\end{split}
	\end{equation}
	provided that
	\begin{align}
		\label{Holder I221 1} & \frac{1}{q_1} + \frac{1}{q_2} = \frac{1}{2}, \\
		\label{Holder I221 2} & \frac{1}{r_1} + \frac{1}{r_2} = \frac{1}{2} \; \text{ or equivalently } \delta(r_1) + \delta(r_2) = \frac{d}{2}.
	\end{align}
	The last term of \eqref{I221} is bounded by $\norm{u}_2$, the second term is treated by using the Lemma \ref{lemma ineq schr} that leads to the following restrictions
	\begin{align*}
		& \frac{2}{q_2} = 1- \eta(1-\gamma) \frac{c_1}{b_0}, \\
		& \delta(r_2) = (1-\eta)(1-\gamma) \frac{c_1}{b_0}.
	\end{align*}
	Using the Sobolev's embedding theorem, the first term of \eqref{I221} is bounded by $\norm{\mathcal{F}^{-1} \left(  \bra{\tau +  |\xi|^2}^{-b_1} |\widehat{w}| \right)}_{L_t^{q_1}L_x^{r_1'}}$, provided that
	\begin{equation}
		\label{Sobolev I221}
		1 \geq \delta(r_1) - \delta(r_1').
	\end{equation}
	Then, we can use the Lemma \ref{lemma ineq schr} with 
	\begin{align*}
		& \frac{2}{q_1} = 1- \eta(1-\gamma) \frac{b_1}{b_0}, \\
		& \delta(r_1')= (1-\eta)(1-\gamma) \frac{b_1}{b_0}.
	\end{align*}
	Therefore, the restrictions \eqref{Holder I221 1}, \eqref{Holder I221 2} are equivalent to
	\begin{align}
		\label{Condition I221 1} & \eta(1-\gamma) = b_0,\\
		\label{Condition I221 2} & 1+ \frac{(1-\eta)(1-\gamma)}{b_0} \geq \frac{d}{2}.
	\end{align}
	We see that \eqref{Condition I221 1} and \eqref{Condition I221 2} lead to $\eta \leq \frac{2}{d}$. That suggests us to take
	\[
	\eta = \frac{2}{d},
	\]
	then
	\[
	1-\gamma = \frac{b_0 d}{2}.
	\]
	If we take $b_0=b_1$ then the constrain $1-\gamma < 1$ implies
	\begin{equation}
		\label{auxi 2}	 b_1 < \frac{2}{d}.
	\end{equation}
	Therefore,
	\begin{equation}
		\label{final I221} |I_{221}| \lesssim  T^{\theta_{221}} \norm{w}_2 \norm{v_1}_2 \norm{u}_2,
	\end{equation}
	with 
	\begin{equation}
		\label{t I221}  \theta_{221} = (1- \frac{b_1 d}{2}) (\frac{3}{2} - b_1).
	\end{equation}

	
	For $I_{222}$, using the H\"older inequality we have
	\begin{equation}
		\label{I222}
		\begin{split}
			I_{222} & \leq   \norm{\mathcal{F}^{-1}\left(\bra{\tau \pm |\xi|}^{-b_2} |\widehat{u}|\right)}_{L_t^{q_1}L_x^{2}} \norm{\mathcal{F}^{-1}\left( \bra{\xi}^{-1} \bra{\tau +  |\xi|^2}^{b_2-b_1} |\widehat{w}|\right)}_{L_t^{q_2}L_x^{r_2}} \\
			& \quad \norm{\mathcal{F}^{-1} \left( \bra{\tau +  |\xi|^2}^{-c_1} |\widehat{v_1}| \right)}_{L_t^{q_3}L_x^{r_3}}
		\end{split}
	\end{equation}
	provided that
	\begin{align}
		\label{Holder I222 1} \frac{1}{q_1} + \frac{1}{q_2} + \frac{1}{q_3} =1, \\
		\label{Holder I222 2} \delta(r_2) + \delta(r_3) = \frac{d}{2}.
	\end{align}
	For the second term of \eqref{I222}, using the Sobolev embedding theorem we have
	\[
	\norm{\mathcal{F}^{-1}\left( \bra{\xi}^{-1} \bra{\tau +  |\xi|^2}^{b_2-b_1} |\widehat{w}|\right)}_{L_t^{q_2}L_x^{r_2}} \lesssim \norm{\mathcal{F}^{-1}\left(  \bra{\tau +  |\xi|^2}^{b_2-b_1} \right)}_{L_t^{q_2}L_x^{r_2'}},
	\]
	if
	\begin{equation}
		\label{Sobolev I222}
		1 \geq \delta(r_2) - \delta(r_2').
	\end{equation}
	$\norm{\mathcal{F}^{-1}\left(  \bra{\tau +  |\xi|^2}^{b_2-b_1} \right)}_{L_t^{q_2}L_x^{r_2'}}$ and the first and the last terms of \eqref{I222} are estimated by using Lemma \ref{lemma ineq schr} provided that
	\begin{align*}
		& \frac{2}{q_1} = 1- (1-\gamma) \frac{b_2}{b_0}, \\
		& \frac{2}{q_2} = 1- \eta(1-\gamma) \frac{b_1-b_2}{b_0}, \\
		& \delta(r_2')= (1-\eta)(1-\gamma) \frac{b_1-b_2}{b_0}, \\
		& \frac{2}{q_3} = 1- \eta(1-\gamma) \frac{c_1}{b_0}, \\
		& \delta(r_3) = (1-\eta)(1-\gamma) \frac{c_1}{b_0}.
	\end{align*}
	Therefore \eqref{Holder I222 1}, \eqref{Holder I222 2} and \eqref{Sobolev I222} become
	\begin{align}
		\label{Condition I222 1} & (1-\gamma) \left(  (1-\eta)b_2 + \eta \right) =b_0,\\
		\label{Condition I222 2} & 1 + (1-\eta)(1-\gamma) \frac{1-b_2}{b_0} \geq \frac{d}{2}.
	\end{align}
	\eqref{Condition I222 1} and \eqref{Condition I222 2} lead to $\eta \leq \frac{2-db_2}{d(1-b_2)}$. That suggests us to take
	\[
	\eta= \frac{2-db_2}{d(1-b_2)},
	\]
	then
	\[
	1-\gamma = \frac{d b_0}{2}.
	\]
	If we take $b_0=b_1$ then we only need to verify $1-\gamma < 1$ that requires
	\begin{equation*}
		b_1 < \frac{2}{d},
	\end{equation*}
	that is exactly \eqref{auxi 2}. Hence
	\begin{equation}
		\label{final I 222}
		I_{222} \lesssim  T^{\theta_{222}} \norm{u}_2 \norm{w}_2 \norm{v_1}_2,
	\end{equation}
	where
	\begin{equation}
		\label{t I222} \theta_{222} = (1-\frac{d b_1}{2}) ( 1- [b_1-b_2 - 1/2]_+).
	\end{equation}

	For $I_{223}$, using the H\"older inequality we get
	\begin{equation}
		\label{I223}
		\begin{split}
			 I_{223}		& \quad \leq  \norm{\mathcal{F}^{-1} \left(  \bra{\tau \pm |\xi|}^{-b_2} |\widehat{u}| \right)}_{L_t^{q_1} L_x^2} \norm{\mathcal{F}^{-1}\left( \bra{\xi}^{-1} \bra{\tau +  |\xi|^2}^{-b_1} |\widehat{w}| \right)}_{L_t^{q_2} L_x^{r_2}} \\
			& \qquad \quad \norm{\mathcal{F}^{-1} \left(   \bra{\tau +  |\xi|^2}^{b_2-c_1} |\widehat{v_1}| \right)}_{L_t^{q_3} L_x^{r_3}},
		\end{split}
	\end{equation}
	provided that
	\begin{align}
		\label{Holder I223 1} & \frac{1}{q_1} + \frac{1}{q_2} + \frac{1}{q_3} =1, \\
		\label{Holder I223 2} & \delta(r_2) + \delta(r_3) = \frac{d}{2}.
	\end{align}
	We continue as previous part, by the Sobolev's embedding theorem
	\[
	\norm{\mathcal{F}^{-1}\left( \bra{\xi}^{-1} \bra{\tau +  |\xi|^2}^{-b_1} |\widehat{w}| \right)}_{L_t^{q_2} L_x^{r_2}} \leq \norm{\mathcal{F}^{-1}\left( \bra{\tau +  |\xi|^2}^{-b_1} |\widehat{w}| \right)}_{L_t^{q_2} L_x^{r_2'}},
	\]
	provided that
	\begin{equation}
		\label{Sobolev I223}
		1 \geq \delta(r_2)-\delta(r_2').
	\end{equation}
	Then the use of Lemma \ref{lemma ineq schr} leads to the following restrictions
	\begin{align*}
		& \frac{2}{q_1}= 1- (1-\gamma) \frac{b_2}{b_0}, \\
		& \frac{2}{q_2} = 1- \eta(1-\gamma) \frac{b_1}{b_0}, \\
		& \delta(r_2') = (1-\eta)(1-\gamma) \frac{b_1}{b_0}, \\
		& \frac{2}{q_3} = 1- \eta(1-\gamma ) \frac{c_1-b_2}{b_0}, \\
		& \delta(r_3)= (1-\eta)(1-\gamma) \frac{c_1-b_2}{b_0}.
	\end{align*}
	The conditions \eqref{Holder I223 1}-\eqref{Holder I223 2} and \eqref{Sobolev I223} then become
	\begin{align}
		\label{Condition I223 1} & (1-\gamma) \left(  b_2 + \eta(1-b_2)\right) =b_0, \\
		\label{Condition I223 2} & 1+(1-\eta)(1-\gamma) \frac{1-b_2}{b_0} \geq \frac{d}{2}.
	\end{align}
	With the same argument as for $I_{222}$, we can take
	\[
	\eta= \frac{2-db_2}{d(1-b_2)}, \, 1- \gamma = \frac{d b_1}{2},
	\]
	with the following condition in $b_1, c_1, b_2$
	\begin{equation}
		\label{auxi 3}
		\left \{
		\begin{split}
			& b_1 < \frac{2}{d}, \\
			& b_2 < c_1 = 1-b_1.
		\end{split}
		\right.
	\end{equation}

	Hence
	\begin{equation}
		\label{final I 223}
		I_{223} \lesssim  T^{\theta_{223}} \norm{u}_2 \norm{w}_2 \norm{v_1}_2,
	\end{equation}
	where
	\begin{equation}
		\label{t I223}  \theta_{223}= (1-\frac{d b_1}{2}) (\frac{3}{2} -b_1).
	\end{equation}
	
	\vspace{2mm}
	
	Using \eqref{final I221}, \eqref{final I 222}, \eqref{final I 223} we summarize the estimate for $I_{22}$.
	\begin{equation}
		\label{final I22}
		I_{22} \lesssim  T^{\theta_{22}} \norm{u}_2 \norm{w}_2 \norm{v_1}_2,
	\end{equation}
	where
	\begin{equation}
		\label{t I22}
		\theta_{22} = \min(\theta_{221}, \theta_{222}, \theta_{223}).
	\end{equation}
	Which is strictly positive with the suitable choice of \(b_1, b_2\).

	\textbf{Estimate $I_{23}$:} We have
	\begin{equation}
		\label{I23} \begin{split}
			|I_{23}| \leq & \norm{\mathcal{F}^{-1} \left(\bra{\xi}^{-k_2} \bra{\tau \pm |\xi|}^{-b_2} |\widehat{u}|\right)}_{L_t^{q_1}L_x^{r_1}} \\
			& \norm{\mathcal{F}^{-1} \left( \bra{\tau +  |\xi|^2}^{-b_1} |\widehat{w}| \right)}_{L_t^{q_2} L_x^{r_2}} \\
			&  \norm{\mathcal{F}^{-1} \left( \bra{\tau +  |\xi|^2}^{-c_1} |\widehat{v_1}| \right)}_{L_t^{q_3} L_x^{r_3}}, 
		\end{split}
	\end{equation}
	with
	\begin{align}
		\label{Holder I23 1} & \frac{1}{q_1} + \frac{1}{q_2} + \frac{1}{q_3} = 1,\\
		\label{Holder I23 2} & \delta(r_1) + \delta(r_2) + \delta(r_3) = \frac{d}{2}.
	\end{align}
	Using the Sobolev's embedding theorem we can estimate the first term of \eqref{I23} as follows
	\[
	\norm{\mathcal{F}^{-1} \left(\bra{\xi}^{-k_2} \bra{\tau \pm |\xi|}^{-b_2} |\widehat{u}|\right)}_{L_t^{q_1}L_x^{r_1}} \leq \norm{\mathcal{F}^{-1} \left( \bra{\tau \pm |\xi|}^{-b_2} |\widehat{u}|\right)}_{L_t^{q_1}L_x^{2}}  ,
	\]
	provided that
	\[
	k_2 \geq \delta(r_1) - \delta(2) = \delta(r_1).
	\]
	Next, we use Lemma \ref{lemma ineq schr}, that leads to the following conditions
	\begin{align*}
		& \frac{2}{q_1} = 1- (1-\gamma) \frac{b_2}{b_0}, \\
		& \frac{2}{q_2} = 1- \eta(1-\gamma) \frac{b_1}{b_0}, \\
		& \delta(r_2) = (1-\eta)(1-\gamma) \frac{b_1}{b_0}, \\
		& \frac{2}{q_3} = 1-\eta(1-\gamma) \frac{c_1}{b_0}, \\
		& \delta(r_3) = (1-\eta)(1-\gamma) \frac{c_1}{b_0}.
	\end{align*}
	Then \eqref{Holder I23 1} and \eqref{Holder I23 2} imply that
	\begin{align}
		\label{Condition I23 1} & (1-\gamma)(b_2 + \eta) = b_0, \\
		\label{Condition I23 2} & \frac{(1-\eta)(1-\gamma)}{b_0} \geq\frac{d}{2} + 2b_2 -1.
	\end{align}
	We can see that \eqref{Condition I23 1}-\eqref{Condition I23 2} are exactly \eqref{Condition I21 1}-\eqref{Condition I21 2}, so we have the following estimate of $I_{23}$
	\begin{equation}
		\label{final I23}
		|I_{23}| \lesssim  T^{\theta_{23}} \norm{v_1}_2 \norm{w}_2 \norm{u}_2,
	\end{equation}
	where
	\begin{equation}
		\label{t I23}  \theta_{23} = (1-\frac{b_1(d+4b_2)}{2+2b_2}) (b_2+\frac{3}{2} -b_1).
	\end{equation}

	\textbf{Estimate $I_{24}$:} Using \eqref{ineq 2} we have
	\begin{equation}
		\label{I24}
		\begin{split}
			I_{24} \leq & \int_{|\xi |> 2|\xi - \xi_1|} \frac{\left(  \bra{\tau_1 \pm |\xi_1|}^{b_2} + \bra{\tau - \tau_1 + |\xi - \xi_1|^2}^{b_2} + \bra{\tau +  |\xi|^2}^{b_2}  \right) \widehat{u} \widehat{w} \widehat{v_1}}{ \bra{\xi_1}^{1-2b_2} \bra{\xi - \xi_1}^{2b_2} \bra{\tau_1 \pm |\xi_1|}^{b_2} \bra{\tau - \tau_1 + |\xi -\xi_1|^2}^{b_1} \bra{\tau +  |\xi|^2}^{c_1}  }, \\
			& \leq I_{241} + I_{242} + I_{243}, 
		\end{split}
	\end{equation}
	where
	\begin{align*}
		& I_{241} = \int_{|\xi |> 2|\xi - \xi_1|} \frac{\widehat{u} \widehat{w} \widehat{v_1}}{\bra{\xi_1}^{1-2b_2} \bra{\xi -\xi_1}^{2b_2} \bra{\tau - \tau_1 +  |\xi - \xi_1|^2}^{b_1} \bra{\tau + |\xi|^2}^{c_1}}, \\
		& I_{242} = \int_{|\xi |> 2|\xi - \xi_1|} \frac{\widehat{u} \widehat{w} \widehat{v_1}}{ \bra{\xi_1}^{1-2b_2} \bra{\xi -\xi_1}^{2b_2} \bra{\tau_1 \pm |\xi_1|}^{b_2} \bra{\tau - \tau_1 +  |\xi - \xi_1|^2}^{b_1-b_2} \bra{\tau + |\xi|^2}^{c_1} }, \\
		&  I_{243} = \int_{|\xi |> 2|\xi - \xi_1|} \frac{\widehat{u} \widehat{w} \widehat{v_1}}{ \bra{\xi_1}^{1-2b_2} \bra{\xi -\xi_1}^{2b_2} \bra{\tau_1 \pm |\xi_1|}^{b_2} \bra{\tau - \tau_1 +  |\xi - \xi_1|^2}^{b_1} \bra{\tau + |\xi|^2}^{c_1-b_2} }
	\end{align*}
	The estimates for $I_{24}$ are essentially the same as for $I_{22}$ with slight modifications. However, for completeness, we will show here the proof of estimates for $I_{24}$.
	
	\vspace*{2mm}
	For $I_{241}$, using the H\"older inequality we get
	\begin{equation}
		\label{I241}
		\begin{split}
			I_{241}  \leq &  \norm{\mathcal{F}^{-1} \left( \bra{\xi}^{1-2b_2} |\widehat{u}| \right)} _{L_t^2L_x^{r_1}} \\
			& \norm{\mathcal{F}^{-1} \left( \bra{\xi}^{-2b_2} \bra{\tau +  |\xi|^2}^{-b_1} |\widehat{w}| \right)}_{L_t^{q_2} L_x^{r_2}} \\
			& \norm{\mathcal{F}^{-1} \left( \bra{\tau +  |\xi|^2}^{-c_1} |\widehat{v_1}| \right)}_{L_t^{q_3} L_x^{r_3}},
		\end{split}
	\end{equation}
	with the H\"older conditions
	\begin{align}
		\label{Holder I241 1} & \frac{1}{q_2} + \frac{1}{q_3} = \frac{1}{2}, \\
		\label{Holder I241 2} & \delta(r_1) + \delta(r_2) + \delta(r_3) = \frac{d}{2}.
	\end{align}
	
	We use the Sobolev's embedding theorem to treat the first two terms of \eqref{I241}
	\begin{align*}
		& \norm{\mathcal{F}^{-1} \left( \bra{\xi}^{1-2b_2} |\widehat{u}| \right)} _{L_t^2L_x^{r_1}}  \leq \norm{u}_2, \\
		& \norm{\mathcal{F}^{-1} \left( \bra{\xi}^{-2b_2} \bra{\tau +  |\xi|^2}^{-b_1} |\widehat{w}| \right)}_{L_t^{q_2} L_x^{r_2}} \leq \norm{\mathcal{F}^{-1} \left(  \bra{\tau +  |\xi|^2}^{-b_1} |\widehat{w}| \right)}_{L_t^{q_2} L_x^{r_2'}},
	\end{align*}
	provided that
	\begin{align*}
		& 1- 2b_2 \geq \delta(r_1),\\
		& 2b_2 \geq \delta(r_2) - \delta(r_2').
	\end{align*}
	Next, we use  Lemma \ref{lemma ineq schr} to estimate $\norm{\mathcal{F}^{-1} \left(  \bra{\tau +  |\xi|^2}^{-b_1} |\widehat{w}| \right)}_{L_t^{q_2} L_x^{r_2'}}$ and $\norm{\mathcal{F}^{-1} \left( \bra{\tau +  |\xi|^2}^{-c_1} |\widehat{v_1}| \right)}_{L_t^{q_3} L_x^{r_3}}$, that leads to the following conditions
	\begin{align*}
		&  \frac{2}{q_2} = 1- \eta(1-\gamma) \frac{b_1}{b_0}, \\
		& \delta(r_2') = (1-\eta)(1-\gamma) \frac{b_1}{b_0}, \\
		& \frac{2}{q_3} = 1- \eta (1-\gamma ) \frac{c_1}{b_0}, \\
		& \delta(r_3) = (1-\eta)(1-\gamma) \frac{c_1}{b_0}.
	\end{align*}
	\eqref{Holder I241 1} and \eqref{Holder I241 2} then become
	\begin{align}
		\label{Condition I241 1}  & \eta (1-\gamma) = b_0, \\
		\label{Condition I241 2} &  \frac{(1-\eta)(1-\gamma)}{b_0} \geq \frac{d}{2} -1.
	\end{align}
	Now, we can see that \eqref{Condition I241 1}-\eqref{Condition I241 2} are exactly \eqref{Condition I221 1}-\eqref{Condition I221 2}, so similarly we can take
	\[
	\eta = \frac{2}{d} \, \text{ and } 1-\gamma = \frac{b_1 d}{2}.
	\] 
	And, therefore,
	\begin{equation}
		\label{final I241} |I_{241}| \lesssim  T^{\theta_{241}} \norm{w}_2 \norm{v_1}_2 \norm{u}_2,
	\end{equation}
	with 
	\begin{equation}
			\label{t I241}  \theta_{241} = (1- \frac{b_1 d}{2}) (\frac{3}{2} - b_1).
	\end{equation}
	
	\vspace{2mm}
	Estimates for $I_{242}$ and $I_{243}$ are the same as the estimates for $I_{222}$ and $I_{223}$ respectively. So, we only show here the main results.
	
	For $I_{242}$,
	\begin{equation}
		\label{final I 242}
		I_{242} \lesssim  T^{\theta_{242}} \norm{u}_2 \norm{w}_2 \norm{v_1}_2,
	\end{equation}
	where
	\begin{equation}
		\label{t I242} \theta_{242} = (1-\frac{d b_1}{2}) ( 1- [b_1-b_2 - 1/2]_+).
	\end{equation}
	
	For $I_{243}$,
	\begin{equation}
		\label{final I 243}
		I_{243} \lesssim  T^{\theta_{243}} \norm{u}_2 \norm{w}_2 \norm{v_1}_2,
	\end{equation}
	where
	\begin{equation}
		\label{t I243}  \theta_{243}= (1-\frac{d b_1}{2}) (\frac{3}{2} -b_1).
	\end{equation}

	
	\textbf{Estimate} $I_4$. Using the Schwartz inequality, we have
	\begin{align*}
		I_4 & = \int \frac{|\xi|\bra{\xi}^{k_2}\widehat{w} \widehat{\overline{w}} \widehat{v_2} }{ \bra{\xi_1} \bra{\xi -\xi_1} \bra{\tau- \tau_1 + |\xi - \xi_1|^2}^{b_1} \bra{\tau_1 - |\xi_1|^2}^{b_1} \bra{\tau \pm |\xi|}^{c_2}} \\
		& \leq \int \frac{\bra{\xi}^{k_2}\widehat{w} \widehat{\overline{w}} \widehat{v_2} }{  \bra{\tau- \tau_1 + |\xi - \xi_1|^2}^{b_1} \bra{\tau_1 - |\xi_1|^2}^{b_1} \bra{\tau \pm |\xi|}^{c_2}}.
	\end{align*}
	Then, \eqref{ineq 3} gives us
	\[
	\bra{\xi}^{k_2} \lesssim \bra{\tau - \tau_1 +  |\xi - \xi_1|^2}^{k_2/2} + \bra{\tau_1 -  |\xi_1|^2}^{k_2/2} + \bra{\tau \pm |\xi|}^{k_2/2},
	\]
	or
	\[
	 \bra{\xi}^{2c_2-1} \lesssim \bra{\tau - \tau_1 +  |\xi - \xi_1|^2}^{c_2-1/2} + \bra{\tau_1 -  |\xi_1|^2}^{c_2-1/2} + \bra{\tau \pm |\xi|}^{c_2-1/2}.
	\]
	Thus,
	\begin{align*}
		I_4 & \lesssim \int \frac{ \left( \bra{\tau - \tau_1 +  |\xi - \xi_1|^2}^{c_2-1/2} + \bra{\tau_1 -  |\xi_1|^2}^{c_2-1/2} + \bra{\tau \pm |\xi|}^{c_2-1/2} \right) \widehat{w} \widehat{\overline{w}}  \widehat{v_2}}{\bra{\tau- \tau_1 + |\xi - \xi_1|^2}^{b_1} \bra{\tau_1 - |\xi_1|^2}^{b_1} \bra{\tau \pm |\xi|}^{c_2} } \\
		& \lesssim I_{41} + I_{42}.
	\end{align*}
	Where
	\begin{align*}
		& I_{41} = \int \frac{  \left(  \bra{\tau - \tau_1 +  |\xi - \xi_1|^2}^{c_2-1/2} + \bra{\tau_1 -  |\xi_1|^2}^{c_2-1/2} \right)  \widehat{w} \widehat{\overline{w}}  \widehat{v_2} }{\bra{\tau- \tau_1 + |\xi - \xi_1|^2}^{b_1} \bra{\tau_1 - |\xi_1|^2}^{b_1} \bra{\tau \pm |\xi|}^{c_2}}, \\
		& I_{42} = \int \frac{  \widehat{w} \widehat{\overline{w}}  \widehat{v_2} }{\bra{\tau- \tau_1 + |\xi - \xi_1|^2}^{b_1} \bra{\tau_1 - |\xi_1|^2}^{b_1} \bra{\tau \pm |\xi|}^{1/2}}.
	\end{align*}
	
	\vspace*{2mm}
	$I_{41}$ involves two terms however from previous estimates we can see that they lead to the same estimate. Thus, using the H\"older inequality we get
	\begin{align}
		\label{I41}
		I_{41} \lesssim  &  \norm{\mathcal{F}^{-1} \left(   \bra{\tau + |\xi|^2}^{-(b_1 -c_2 + 1/2)} |\widehat{w}|  \right)}_{L_t^{q_1} L_x^{r_1}}  \norm{\mathcal{F}^{-1} \left( \bra{\tau +  |\xi|^2}^{-b_1} |\widehat{w}| \right)}_{L_t^{q_2} L_x^{r_2}} \\
		& \qquad \norm{\mathcal{F}^{-1} \left(  \bra{\tau \pm |\xi|}^{-c_2} |\widehat{v_2}| \right)}_{L_t^{q_3} L_x^{2}}.
	\end{align}
	Provided that
	\begin{align}
		\label{Holder I41 1} \frac{1}{q_1} + \frac{1}{q_2} + \frac{1}{q_3} =1, \\
		\label{Holder I41 2} \delta(r_1) + \delta(r_2) = \frac{d}{2}.
	\end{align}
	The three terms of \eqref{I41} are estimated by using Lemma \ref{lemma ineq schr}, that leads to the following restrictions
	\begin{align*}
		& \frac{2}{q_1} = 1- \eta (1-\gamma) \frac{b_1-c_2 + 1/2}{b_0}, \\
		& \frac{2}{q_2} = 1- \eta (1-\gamma) \frac{b_1}{b_0},\\
		& \frac{2}{q_3} = 1- (1-\gamma) \frac{c_2}{b_0}, \\
		& \delta(r_1) = (1-\eta)(1-\gamma) \frac{b_1-c_2+1/2}{b_0}, \\
		& \delta(r_2) = (1-\eta)(1-\gamma) \frac{b_1}{b_0}.
	\end{align*}
	Then \eqref{Holder I41 1} and \eqref{Holder I41 2} become
	\begin{align}
		\label{Condition I41 1} & (1-\gamma) \left( \eta (2b_1 -c_2+1/2) +c_2 \right) =b_0, \\
		\label{Condition I41 2} & (1-\eta)(1-\gamma) (2b_1 -c_2 + 1/2) = \frac{d}{2} b_0.
	\end{align}
	So, we can take
	\[
	\eta = \frac{2b_1-(1+d/2) c_2 + 1/2}{(2b_1-c_2+1/2)(d/2+1)},
	\]
	if $b_1$ and $c_2$ (or $b_2$) satisfy 
	\begin{equation}
		\label{auxi 4} 2b_1 + (1+d/2) b_2 > \frac{1+d}{2}.
	\end{equation}
	Then
	\[
	1-\gamma = \frac{b_0(d+2)}{4b_1+1}.
	\]

	Note that we can choose $b_0=b_1$ and the condition $1-\gamma<1$ always holds. Hence,
	\begin{equation}
		\label{final I41}
		|I_{41}| \lesssim  T^{\theta_{41}} |w|_2^2 |v_2|,
	\end{equation}
	where
	\begin{equation}
		\label{t I41}  \theta_{41} = (1-\frac{b_0(d+2)}{4b_1+1}) \left(b_1+b_2 + 1/2 - [b_1+b_2-1]_{+}\right).
	\end{equation}
	
	\vspace*{2mm}
	For $I_{42}$, using the H\"older inequality we get
	\begin{equation}
		\label{I42}
		\begin{split}
			I_{42} & \leq  \norm{\mathcal{F}^{-1}\left(\bra{\tau + |\xi|^2}^{-b_1} |\widehat{w}|\right)}_{L_t^{q_1}L_x^4}^2 \norm{\mathcal{F}^{-1} \left( \bra{\tau \pm |\xi|}^{-1/2} |\widehat{v_2}|\right)}_{L_t^{q_2} L_x^2}. 
		\end{split}
	\end{equation}
	Where
	\begin{align}
		\label{Holder I42 1} \frac{2}{q_1} + \frac{1}{q_2} =1.
	\end{align}
	Using Lemma \ref{lemma ineq schr} we have the following constraints
	\begin{align*}
		& \frac{2}{q_1} = 1- \eta(1-\gamma)\frac{b_1}{b_0}, \\
		& \frac{2}{q_2} = 1- (1-\gamma) \frac{1}{2b_0}, \\
		& \delta(4)= \frac{d}{4} = (1-\eta)(1-\gamma) \frac{b_1}{b_0}.
	\end{align*}
	Hence,
	\begin{align}
		\label{Condition I42 1} & (1-\gamma)(4\eta b_1 +1) = 2b_0, \\
		\label{Condition I42 2} & (1-\eta)(1-\gamma)b_1=b_0 \frac{d}{4}.
	\end{align}
	Thus, we can take 
	\begin{align*}
		& \eta = \frac{8b_1 -d}{4db_1+8b_1}, \\
		& 1-\gamma = \frac{(d+2)b_0}{4b_1 + 1}.
	\end{align*}
	Note that to ensure  $1-\gamma < 1$, we need 
	\[
	b_0 < \frac{4b_1+1}{d+2},
	\]
	so in order to choose $b_0>\frac{1}{2}$, $b_1$ should satisfies
	\[
	\frac{4b_1 + 1}{d+2} > \frac{1}{2}
	\]
	or $b_1 > d/8$ which holds in both cases of $d$. Thus, we can take $b_0=b_1$ in this case.
	
	Therefore,
	\begin{equation}
		\label{final I42} |I_{42}| \lesssim  T^{\theta_{42}} |w|_2^2 |v_2|_2.
	\end{equation}
	Where
	\begin{equation}
		\label{t I42}  \theta_{42} = (1- \frac{(d+2)b_1}{4b_1+1}) \left(\frac{3}{2} - [0]_+\right).
	\end{equation}
	
	\textbf{Estimate $I_5$.} We have
	\begin{align*}
		I_5 & = \int \frac{\xi^{(1)} \tau  \bra{\xi}^{k_2}\widehat{w} \widehat{\overline{w}} \widehat{v_2} }{|\xi| \bra{\xi_1} \bra{\xi -\xi_1} \bra{\tau- \tau_1 + |\xi - \xi_1|^2}^{b_1} \bra{\tau_1 - |\xi_1|^2}^{b_1} \bra{\tau \pm |\xi|}^{c_2}} \\
		& \leq \int \frac{ |\tau|  \bra{\xi}^{k_2}|\widehat{w}| |\widehat{\overline{w}}| |\widehat{v_2}| }{ \bra{\xi_1} \bra{\xi -\xi_1} \bra{\tau- \tau_1 + |\xi - \xi_1|^2}^{b_1} \bra{\tau_1 - |\xi_1|^2}^{b_1} \bra{\tau \pm |\xi|}^{c_2}} \\
		& \leq I_{51} + I_{52},
	\end{align*}
	where
	\[
	\begin{split}
		I_{51}& = \int_{|\tau| <2|\xi|} \frac{ |\tau|  \bra{\xi}^{k_2}|\widehat{w}| |\widehat{\overline{w}}| |\widehat{v_2}| }{ \bra{\xi_1} \bra{\xi -\xi_1} \bra{\tau- \tau_1 + |\xi - \xi_1|^2}^{b_1} \bra{\tau_1 - |\xi_1|^2}^{b_1} \bra{\tau \pm |\xi|}^{c_2}} \\
		& \lesssim |I_4|.
	\end{split}
	\]
	and we only need to estimate
	\[
	\begin{split}
		I_{52} & = \int_{|\tau| \geq 2|\xi|} \frac{ |\tau|  \bra{\xi}^{k_2}|\widehat{w}| |\widehat{\overline{w}}| |\widehat{v_2}| }{ \bra{\xi_1} \bra{\xi -\xi_1} \bra{\tau- \tau_1 + |\xi - \xi_1|^2}^{b_1} \bra{\tau_1 - |\xi_1|^2}^{b_1} \bra{\tau \pm |\xi|}^{c_2}}\\
		& = \int_{|\tau| \geq 2|\xi|} \frac{ |\tau|  \bra{\xi}^{2c_2-1}|\widehat{w}| |\widehat{\overline{w}}| |\widehat{v_2}| }{ \bra{\xi_1} \bra{\xi -\xi_1} \bra{\tau- \tau_1 + |\xi - \xi_1|^2}^{b_1} \bra{\tau_1 - |\xi_1|^2}^{b_1} \bra{\tau \pm |\xi|}^{c_2}}
	\end{split}
	\]
	We observe that, if $|\tau| \geq 2 |\xi|$ then
	\[
	|\tau \pm |\xi|| \geq |\tau| - |\xi| \geq \frac{|\tau|}{2},
	\]
	or 
	\[
	|\tau|^{c_2} \lesssim \bra{\tau \pm |\xi|}^{c_2}.
	\]
	
	That implies
	\[
	I_{52} \leq \int_{|\tau| \geq 2|\xi|} \frac{ |\tau|^{1-c_2}  \bra{\xi}^{2c_2-1}|\widehat{w}| |\widehat{\overline{w}}| |\widehat{v_2}| }{ \bra{\xi_1} \bra{\xi -\xi_1} \bra{\tau- \tau_1 + |\xi - \xi_1|^2}^{b_1} \bra{\tau_1 - |\xi_1|^2}^{b_1}}
	\]
	By the way,   \eqref{ineq 5} tells us
	\[
	\bra{\xi_1}^{2(1-c_2)} \bra{\xi - \xi_1}^{2(1-c_2)} \bra{\tau - \tau_1 +  |\xi - \xi_1|^2}^{1-c_2} \bra{\tau_1 - |\xi_1|^2}^{1-c_2} \gtrsim \bra{\tau}^{1-c_2}.
	\]
	Combining with the Cauchy-Schwartz inequality
	\[
	\bra{\xi_1}^{2c_2-1} + \bra{\xi - \xi_1}^{2c_2-1} \geq \bra{\xi}^{2c_2-1}
	\]
	we obtain
	\[
	I_{52} \leq I_{521} + I_{522}.
	\]
	Where
	\begin{align*}
		& I_{521} = \int_{|\tau| \geq 2|\xi|} \frac{ |\widehat{w}| |\widehat{\overline{w}}| |\widehat{v_2}| }{ \bra{\xi_1}^{2c_2-1}  \bra{\tau- \tau_1 + |\xi - \xi_1|^2}^{b_1+c_2-1} \bra{\tau_1 - |\xi_1|^2}^{b_1+c_2-1}}, \\
		& I_{522} = \int_{|\tau| \geq 2|\xi|} \frac{ |\widehat{w}| |\widehat{\overline{w}}| |\widehat{v_2}| }{  \bra{\xi -\xi_1}^{2c_2-1} \bra{\tau- \tau_1 + |\xi - \xi_1|^2}^{b_1+c_2-1} \bra{\tau_1 - |\xi_1|^2}^{b_1+c_2-1}}.
	\end{align*}
	In our analysis, $I_{521}$ and $I_{522}$ are similar so we consider only the estimate for $I_{521}$.\\
	Using the H\"older inequality we get
	\begin{equation}
		\label{I521} 
		\begin{split}
			|I_{521}| &\leq \norm{\mathcal{F}^{-1} \left( \bra{\xi}^{-(2c_2-1)}\bra{\tau + |\xi|^2}^{-(b_1+c_2-1)} |\widehat{w}| \right)}_{L_t^{4}L_x^{r_1}}\\
			& \quad  \norm{\mathcal{F}^{-1} \left( \bra{\tau +  |\xi|^2}^{-(b_1+c_2-1)} |\widehat{w}| \right)}_{L_t^{4} L_x^{r_2}} \norm{v_2}_2.
		\end{split}
	\end{equation}
	For the convenience, we rewrite \eqref{I521} using the notation of $b_2$ as follows
	\begin{equation}
		\label{I521_2} 
		\begin{split}
			|I_{521}| &\leq \norm{\mathcal{F}^{-1} \left( \bra{\xi}^{-(1-2b_2)}\bra{\tau + |\xi|^2}^{-(b_1-b_2)} |\widehat{w}| \right)}_{L_t^{4}L_x^{r_1}}\\
			& \quad  \norm{\mathcal{F}^{-1} \left( \bra{\tau +  |\xi|^2}^{-(b_1-b_2)} |\widehat{w}| \right)}_{L_t^{4} L_x^{r_2}} \norm{v_2}_2.
		\end{split}
	\end{equation}
	$r_1$ and $r_2$ then satisfy 
	\begin{equation}
		\label{Holder I521} \delta(r_1) + \delta(r_2) = \frac{d}{2}.
	\end{equation}
	The first term of \eqref{I521_2} is estimated by using the Sobolev's embedding 
	\[
	W^{1-2b_2,r_1'} \hookrightarrow L_x^{r_1},
	\]
	provided that
	\[
	1-2b_2 > \delta(r_1) - \delta(r_1').
	\]
	Then, we can process as in previous parts that uses the lemma \ref{lemma ineq schr} and leads to the following constrains
	\begin{align*}
		& \frac{1}{2} = 1- \eta(1-\gamma) \frac{b_1-b_2}{b_0}, \\
		& \delta(r_1')= \delta(r_2) >\delta(r_1) + 2b_2-1.
	\end{align*}
	That is equivalent to
	\begin{align}
		\label{Condition I521 1} & \eta(1-\gamma)(b_1-b_2) = \frac{b_0}{2}, \\
		\label{Condition I521 2} & 2(1-\eta)(1-\gamma) \frac{b_1-b_2}{b_0} > \frac{d}{2} + 2b_2-1.
	\end{align}
	Combining \eqref{Condition I521 1}-\eqref{Condition I521 2} we obtain
	\[
	\eta \leq \frac{1}{2b_2 + d/2}.
	\]
	That suggests us to take $\eta = \frac{1}{2b_2 + d/2}$, then 
	\[
	1- \gamma = \frac{b_0 (2b_2 + d/2)}{2(b_1-b_2)}.
	\]
	We have that $b_1>1/2>b_2$ so it remains to verify that we can choose $b_0> 1/2$ so that $1-\gamma < 1$ and $(1-\gamma)(b_1-b_2) \leq b_0$.
	
	The constrain $1-\gamma < 1$ requires
	\[
	b_0 < \frac{2(b_1-b_2)}{2b_2 + d/2},
	\]
	thus, $b_1,b_2$ must satisfy
	\[
	\frac{2(b_1-b_2)}{2b_2 + d/2} > \frac{1}{2},
	\]
	or
	\begin{equation}
		\label{I521 auxi 1}
		b_2<\frac{2}{3}b_1- \frac{d}{12}.
	\end{equation}
	
	Combining  \eqref{Condition I521 1} with constrain $(1-\gamma)(b_1-b_2) \leq b_0$ leads to 
	\[
	\eta \geq \frac{1}{2}
	\]
	or equivalently
	\begin{equation}
		\label{I521 auxi 2}
		b_2 \leq 1- \frac{d}{4}.
	\end{equation}
	From \eqref{I521 auxi 1}, \eqref{I521 auxi 2} and the constrain $b_1>1/2$, we require that
	\begin{equation}
		\label{I521 auxi 3}
		\left\{
		\begin{split}
			& b_2 < \frac{1}{6} \, \text{ if } d=2, \\
			& b_2< \frac{1}{12} \, \text{ if } d=3.
		\end{split}
		\right.
	\end{equation}
	Therefore, we have
	\begin{equation}
		\label{final I521}
		|I_{521}| \lesssim  T^{\theta_{521}} \norm{w}_2^2 \norm{v_2}_2,
	\end{equation}
	where
	\begin{equation}
		\label{t I521}  \theta_{521} = 2 (1- \frac{b_0 (2b_2 + d/2)}{2(b_1-b_2)}) (b_1-b_2) \left( 1- \frac{[b_1-b_2-1/2]_+}{b_1-b_2}\right).
	\end{equation}
	
	\subsection{Proof of the main theorem}  \label{subsection proof}
	\begin{proof}
		We are going to determine the condition of $b_1$ and $b_2$. Let recall that \(k_2= 1-2b_2\) so that the range of \(b_2\) defines the range of \(k_2\) or \(l\) in Theorem \ref{theorem main}. In other hand, since we fix the order of Sobolev space for \(\psi\) then \(b_1\) can be chosen more freely so that all the condition hold.
	
	Combining \eqref{auxi 1},\eqref{auxi 2}, \eqref{auxi 3}, \eqref{auxi 4} and \eqref{I521 auxi 3} we have
	\begin{align*}
		& b_2<1-b_1,\\
		& b_1 < \frac{2+2b_2}{d+ 4b_2}, \\
		& b_1<2/d, \\
		& b_2<\frac{1}{6} \, \text{ if } d=2, \quad b_2<\frac{1}{12} \, \text{ if } d=3, \\
		& 2b_1 + (1+d/2)b_2 > \frac{d+1}{2}.
	\end{align*}
	Therefore, we can conclude the conditions for \(b_1, b_2\) as follows.\\
	 For $d=2$, 
	\begin{equation}
		\label{d=2 auxi condition} \left \{
		\begin{split}
			& \frac{3}{4} < b_1 < \frac{5}{6}, \\
			& 0 \leq b_2 <\frac{1}{6} \text{ or } \frac{2}{3} < k_2 \leq 1.
		\end{split}
		\right.
	\end{equation}
	For $d=3$,
	\begin{equation}
		\label{d=3 auxi condition} \left \{
		\begin{split}
			& \frac{1}{2} < b_1 < \frac{13}{20}, \\
			& 0 < b_2 <\frac{1}{12} \text{ or } \frac{5}{6} < k_2 <1.
		\end{split}
		\right.
	\end{equation}
	Those conditions combining with our argument explanation finish the proof of Theorem \ref{theorem main}.
	\end{proof}
	\section{Conclusion and open questions} \label{Section Conclusion}
	\begin{itemize}
		\item [i)] Our result basically improves the regularity condition in the local Cauchy problem for the Zakharov-Rubenchik system in \(2 \text{ or } 3\) dimension that was studied in \cite{Ponce2005} and \cite{Luong2018}. The proof is based on the derivation of corresponding Bourgain spaces and carefully estimations of the terms involved.
		\item [ii)] The result also strengthens the global weak solution obtained by extending the local solution under certain condition of parameters of the system since at least the first component \(\psi\) lies in the energy space. We however, not able to reach the same goal with \(\rho \text{ and } \phi\) due to the technical difficulties.
		\item [iii)] This paper is also our preparation in more important problem where we take into account the ``model parameter'' $\epsilon$ and expect to get the existence time of order $O(1/\epsilon^\alpha)$ with \(\alpha>0\).
		
		 It is also interesting to study the original Benney-Roskes system with ``full-dispersion'' derived in \cite{Lannes2013} where the Schr\"odinger operator is replaced by 
	\begin{equation*}
		\frac{\omega(\textbf{k} + \epsilon D) - \omega(\textbf{k})}{\epsilon},
	\end{equation*}
	the dispersion \(\omega\) is given by 
	\begin{equation*}
		\omega(\xi) = \left(|\xi| \tanh (\sqrt{\mu} |\xi|)\right)^{1/2},
	\end{equation*}
	where \(\epsilon, \mu\) are model-parameters.\\ 
	It would be possible if one could derive a ``Strichartz type'' estimate for the full dispersion operator.
	\end{itemize}
	\bibliographystyle{amsplain}

\begin{thebibliography}{10}
	
	\bibitem{Ginibre1997}
	J.~Ginibre, Y.~Tsutsumi, and G.~Velo, \emph{On the {C}auchy problem for the
		{Z}akharov system}, J. Funct. Anal. \textbf{151} (1997), no.~2, 384--436.
	\MR{1491547}
	
	\bibitem{Lannes2013}
	David Lannes, \emph{The water waves problem}, Mathematical Surveys and
	Monographs, vol. 188, American Mathematical Society, Providence, RI, 2013,
	Mathematical analysis and asymptotics. \MR{3060183}
	
	\bibitem{Linares2009}
	Felipe Linares and Carlos Matheus, \emph{Well posedness for the 1{D}
		{Z}akharov-{R}ubenchik system}, Adv. Differential Equations \textbf{14}
	(2009), no.~3-4, 261--288. \MR{2493563}
	
	\bibitem{Luong2018a}
	Hung Luong, \emph{Local well-posedness for the {Z}akharov system on the
		background of a line soliton}, Commun. Pure Appl. Anal. \textbf{17} (2018),
	no.~6, 2657--2682. \MR{3814393}
	
	\bibitem{Luong2018}
	Hung Luong, Norbert~J. Mauser, and Jean-Claude Saut, \emph{On the {C}auchy
		problem for the {Z}akharov-{R}ubenchik/{B}enney-{R}oskes system}, Commun.
	Pure Appl. Anal. \textbf{17} (2018), no.~4, 1573--1594. \MR{3842874}
	
	\bibitem{Oliveira2003}
	Filipe Oliveira, \emph{Stability of the solitons for the one-dimensional
		{Z}akharov-{R}ubenchik equation}, Phys. D \textbf{175} (2003), no.~3-4,
	220--240. \MR{1963861}
	
	\bibitem{Passot1996}
	T~Passot, C~Sulem, and PL~Sulem, \emph{Generation of acoustic fronts by
		focusing wave packets}, Physica D: Nonlinear Phenomena \textbf{94} (1996),
	no.~4, 168--187.
	
	\bibitem{Ponce2005}
	Gustavo Ponce and Jean-Claude Saut, \emph{Well-posedness for the
		{B}enney-{R}oskes/{Z}akharov-{R}ubenchik system}, Discrete Contin. Dyn. Syst.
	\textbf{13} (2005), no.~3, 811--825. \MR{2153145}
	
	\bibitem{Sulem1999}
	Catherine Sulem and Pierre-Louis Sulem, \emph{The nonlinear {S}chr\"odinger
		equation}, Applied Mathematical Sciences, vol. 139, Springer-Verlag, New
	York, 1999, Self-focusing and wave collapse. \MR{1696311}
	
	\bibitem{Tzvetkov2018}
	Nikolay Tzvetkov, \emph{Transverse stability issues in {H}amiltonian {PDE}},
	European {C}ongress of {M}athematics, Eur. Math. Soc., Z\"urich, 2018,
	pp.~619--639. \MR{3890445}
	
\end{thebibliography}

\end{document}